\documentclass[12pt]{article}
		\title{Finite symmetry group actions on substitution tiling C*-algebras}
		
		\author{Charles Starling\thanks{Research supported by NSERC. Departamento de Matem\'atica Universidade Federal de Santa Catarina 88040-900 Florian\'opolis SC, Brazil. \texttt{slearch@gmail.com}}}

\usepackage[english]{babel}
\usepackage{graphicx}
\usepackage{latexsym}
\usepackage{amsfonts}
\usepackage{dsfont}
\usepackage{amscd}
\usepackage{amssymb}
\usepackage{amsmath}
\usepackage{mathtools}
\usepackage{xy}
\usepackage{mathrsfs}
\usepackage{amsthm}
\usepackage{tikz}

\input xy
\xyoption{all}

\newcommand{\NN}{\mathbb{N}}

\newcommand{\re}{\mathbb{R}}
\newcommand{\rd}{\mathbb{R}^d}
\newcommand{\bz}{\mathbb{Z}}

\newcommand{\rp}{\mathcal{R}_{\textnormal{punc}}}
\newcommand{\raf}{\mathcal{R}_{AF}}
\newcommand{\op}{\Omega_{\textnormal{punc}}}
\newcommand{\ppen}{\mathcal{P}_{\textnormal{Pen}}}
\newcommand{\Npro}{N_{\text{pro}}}
\newcommand{\p}{\mathcal{P}}
\newcommand{\x}{\mathbf{x}}
\newcommand{\g}{\mathscr{G}}
\newcommand{\eps}{\varepsilon}
\newcommand{\Aw}{A_\omega}

\DeclareMathOperator{\supp}{supp}

\theoremstyle{plain}
\newtheorem{theorem}{Theorem}[section]
\newtheorem{lemma}[theorem]{Lemma}
\newtheorem{corollary}[theorem]{Corollary}
\newtheorem{proposition}[theorem]{Proposition}

\newtheorem{assumption}[theorem]{Assumption}

\theoremstyle{definition}
\newtheorem{rmk}[theorem]{Remark}
\newtheorem{defn}[theorem]{Definition}
\newtheorem{example}[theorem]{Example}

\begin{document}
\maketitle
\begin{abstract}
For a finite symmetry group $G$ of an aperiodic substitution tiling system $(\p,\omega)$, we show that the crossed product of the tiling C*-algebra $\Aw$ by $G$ has real rank zero, tracial rank one, a unique trace, and that order on its K-theory is determined by the trace. We also show that the action of $G$ on $\Aw$ satisfies the weak Rokhlin property, and that it also satisfies the tracial Rokhlin property provided that $\Aw$ has tracial rank zero. In the course of proving the latter we show that $\Aw$ is finitely generated. We also provide a link between $\Aw$ and the AF algebra Connes associated to the Penrose tilings.
\end{abstract}

\section{Introduction}

A substitution tiling is a tiling of $\rd$ formed from a primitive substitution rule $\omega$ on a finite prototile set $\p=\{p_1, p_2, \dots, p_{\Npro}\}$. Most tilings of interest display finite rotational and reflectional symmetries. For instance, the famous Penrose tilings of $\re^2$ display symmetry under rotation by $^\pi\hspace{-0.05cm}/_5$. A finite group $G \subset O(d,\re)$ acting on $\p$ will be called a {\em symmetry group} for $(\p,\omega)$ if it commutes with the substitution $\omega$. Here we study the actions of such groups on C*-algebras associated to $(\p,\omega)$.

Associating a C*-algebra to an aperiodic substitution $(\p,\omega)$ goes back to Connes \cite{CoNCG} who constructed an AF algebra from Penrose tilings. He showed that Penrose tilings are in one-to-one correspondence with infinite paths through a certain Bratteli diagram, and that two such paths are tail equivalent if and only if their associated tilings could be taken to one another by an isometry of the plane. Later, Kellendonk \cite{KelNCG} defined a different C*-algebra $\Aw$ from $(\p,\omega)$; this is the reduced C*-algebra of the \' etale groupoid $\rp$ of translational equivalence only. This C*-algebra has been studied by many authors, see for instance \cite{KP00}, \cite{Pu00}, and \cite{Ph05}.

Since the C*-algebra Connes associated to the Penrose tiling is AF, it is classified by its pointed $K_0$ group by \cite{El76}. Not only is the C*-algebra $\Aw$ not in general an AF algebra, it is not known whether it can be classified by its Elliott invariant (which is essentially K-theory paired with traces). Some partial progress has been made in this direction beginning with work of Putnam \cite{Pu00} who proved that the order on $K_0(\Aw)$ was determined by the unique normalized trace on $\Aw$. To do this, Putnam used the presence of a canonical AF subalgebra $AF_\omega\subset \Aw$ which is the C*-algebra of an AF subequivalence relation $\raf\subset\rp$.

In \cite{Ph05} Phillips generalized the techniques of \cite{Pu00} to the C*-algebras of what he called {\em almost AF Cantor groupoids}. A groupoid $\g$ is of this type if it has an AF subgroupoid $\g_0\subset \g$ that is ``large'' in a suitable sense; this notably includes not only $\rp$ but also groupoids associated to free minimal actions of $\bz^d$ on the Cantor set. Phillips proved that the reduced C*-algebras of these groupoids have some nice classification properties, including real rank zero, stable rank one, and order on $K_0$ being determined by traces. He also showed that normalized traces on $C^*_r(\g)$ are in one-to-one correspondence with normalized traces on $C^*_r(\g_0)$. These results can be seen as significant progress towards answering the question of whether such C*-algebras can be classified by their Elliott invariant. Phillips notes in \cite{Ph05} Question 8.1 that if one could prove that all such C*-algebras have {\em tracial rank zero} then they would be classified by their Elliott invariant, though whether they have tracial rank zero or not is still unresolved.

In this paper we study the action of a finite symmetry group $G$ on the C*-algebras $\Aw$ and $AF_\omega$ under the assumption that $G$ acts freely on $\p$. This situation is far from rare -- indeed, substitutions are typically presented on a finite set of {\em standard position} prototiles and extended by symmetry. Sections 2--4 are background material on tilings, groupoids, and their C*-algebras. Among the background in Section \ref{tilingalgebra} we show  that $\Aw$ is finitely generated (Proposition \ref{Awfinitelygenerated}). In Section \ref{crossedproductsection} we prove that that the crossed product $\Aw\rtimes G$ is isomorphic to the C*-algebra of an almost AF Cantor groupoid, and hence $\Aw\rtimes G$ has real rank zero, stable rank one, and the order on $K_0(\Aw\rtimes G)$ is determined by traces (Theorem \ref{RpxGalmostAF}). In Section \ref{tracesection} we show that $\Aw\rtimes G$ has a unique trace (Corollary \ref{uniquetracecor}), and show in Remark \ref{connesremark} that the AF algebra that Connes associated to the Penrose tilings is isomorphic to $AF_\omega\rtimes D_{10}$ where $D_{10}$ is the dihedral group $D_{10}\subset O(2,\re)$ generated by $r$ (the counterclockwise rotation by $^\pi\hspace{-0.05cm}/_5$) and $f$ (the reflection over the $y$-axis).  

In Section \ref{RPsection} we study properties of the action of $G$ on $\Aw$. We first prove that the action of $G$ on $\Aw$ has the {\em weak Rokhlin property} of Matui and Sato \cite{MaSa12}. We also note that if we assume that $\Aw$ has tracial rank zero, then the action of $G$ on $\Aw$ satisfies the {\em tracial Rokhlin property} of Phillips, see \cite{Ph11}. These results both imply that if one could prove that $\Aw$ has tracial rank zero, then this would imply that $\Aw\rtimes G$ would also have tracial rank zero, and so would be classified by its Elliott invariant. 

\section{Tilings}

A {\em tile} is a subset of $\rd$ homeomorphic to the closed unit ball. A {\em partial tiling} is a collection of tiles whose interiors are pairwise disjoint. A finite partial tiling will be called a {\em patch}. The {\em support} of a partial tiling $P$ is the union of its tiles and is denoted supp$(P)$. We define a {\em tiling} to be a partial tiling whose support is $\rd$. Given $U\subset\rd$ and a partial tiling $T$, $T(U)$ is all the tiles that intersect $U$, that is, $T(U) = \{t\in T \mid t\cap U\neq \emptyset\}$. For $x\in\rd$, $T(\{x\})$ is abbreviated $T(x)$. Two partial tilings $T$ and $T'$ are said to {\em agree on $U$} if $T(U) = T'(U)$. A partial tiling $P$ is called {\em connected} if Int(supp$(P)$) is connected.

Given a vector $x\in\rd$ we can take any subset $U\subset \rd$ and form its translate by $x$, namely $U+x = \{ u+x \mid u\in U\}$. Thus, given a tiling $T$ we can form another tiling by translating every tile by $x$. We denote the new tiling by $T+x = \{t+x \mid t\in T\}$. A tiling for which $T+x = T$ for some non-zero $x\in\rd$ is called {\em periodic}. A tiling for which no such non-zero vector exists is called {\em aperiodic}.

In this paper we deal with substitution tilings. Let $\p = \{p_1, p_2, \dots, p_{\Npro}\}$ be a finite set of (possibly labeled) tiles which we call the set of {\em prototiles}. The prototiles will typically be polytopes, but we assume at minimum that their boundaries have {\em capacity} or {\em box-counting dimension} strictly less than $d$, that is, there exist $n_{\text{cap}}\in\NN$ strictly less than $d$, a constant $K_{\text{cap}}$, and a function $f_{\text{cap}}(\eps) \leq K_{\text{cap}}\eps^{-n_{\text{cap}}}$ such that we can cover the boundary of any prototile by $f_{\text{cap}}(\eps)$ balls of radius $\eps$. By giving the prototiles labels, we may assume that if $p,q\in\p$ and $p+x = q$ then $x=0$. 

Let $\p^*$ be the set of all possible partial tilings consisting of translates of elements of $\p$. A {\em substitution rule} is a function $\omega: \p\to \p^*$ such that there exists $\lambda >1$ such that supp$(\omega(p)) = \lambda{p}$ for all $p\in\p$. We can extend the definition of $\omega$ to tiles which are translates of prototiles by setting $\omega(p+x) = \omega(p) + \lambda x$ for $p\in\p$, $x\in\rd$, and it is not hard to see that this extends $\omega$ to a map from $\p^*$ to $\p^*$. The pair $(\p,\omega)$ is called a {\em substitution tiling system}.

We let $\Omega$ be the set of all tilings $T$ in $\p^*$ such that if $P$ is a patch in $T$ then there exists $x\in\rd$, $p\in\p$ and $n\in\NN$ such that $P\subset \omega^n(p) + x$. It is not hard to show that the set $\Omega$ is nonempty and $\omega(\Omega) = \Omega$ (see for example \cite{AP98}, Propositions 2.1 and 2.2). We make the following assumptions on $(\p,\omega)$:
\begin{assumption}\label{primitive} The substitution tiling system $(\p,\omega)$ is {\em primitive}, that is, there exists $N\in\NN$ such that for all $p, q\in \p$, $\omega^n(p)$ contains a translate of $q$.
\end{assumption}
Primitivity allows construction of a specific type of tiling in $\Omega$ which will be important to us. For $p\in\p$, primitivity allows us to find $n\in \NN$ such that $p+x\in\omega^n(p)$ and $p+x$ is contained in the interior of supp$(\omega^n(p))$. The map from supp$(\omega^n(p))$ to $p+x$ given by $y \mapsto \lambda^{-n}y + x$ is onto and contractive, and hence has a unique fixed point $y_0$, that is, $y_0 = \lambda^{-n}y_0 + x$. If we set $t = p - y_0 + x$, one then checks that $t\in \omega^n(t)$, and so $T = \cup_k \omega^{kn}(t)$ will be a tiling with $\omega^n(T) = T$ and $T(0)$ is a single tile.

\begin{assumption}\label{FLC}
Every element $T\in\Omega$ has {\em finite local complexity}, that is, for every $R>0$ the number of patches $P\subset T$ such that the diameter of supp$(P)$ is less than $R$ is finite modulo translation.
\end{assumption}
\begin{assumption}\label{injectivesub} The map $\omega:\Omega\to\Omega$ is injective.
\end{assumption}
Under these assumptions, every tiling in $\Omega$ is aperiodic (see for example \cite{AP98}, Proposition 2.3) and $\omega$ has an inverse $\omega^{-1}$ such that $\omega^{-1}(T+x) = \omega^{-1}(T) + \lambda^{-1}x$. 
\begin{assumption}\label{forcesborder}
The substitution system $(\p,\omega)$ {\em forces its border}, that is, there exists an $n\in\NN$ such that for all $p\in \mathcal{P}$ if we have that whenever $\omega^n(p) + x \in T$ and $\omega^n(p) + x'\in T'$ then we can conclude that 
\[
T\left(\supp(\omega^n(p) + x)\right) - x = T'\left(\supp(\omega^n(p) + x')\right) - x'.
\]
\end{assumption}
We define a metric on $\Omega$ under which two tilings will be close if they agree on a large ball around the origin up to a small translation. For $T,T'\in\Omega$ we let
\begin{eqnarray*}
d(T,T') & = &  \inf\{1,\varepsilon \mid \exists\ x, x'\in\rd \ni \left\|x\right\|, \left\|x'\right\| < \varepsilon,\\
        &   & \hspace{1cm}(T - x)(B_{1/\varepsilon}(0)) = (T'-x')(B_{1/\varepsilon}(0))\}.
\end{eqnarray*} 
This is called the {\em tiling metric}. The space $\Omega$ equipped with this metric is called the {\em continuous hull}. Finite local complexity implies that $\Omega$ is compact under this metric, and taken together Assumptions \ref{primitive}--\ref{injectivesub} imply that $\omega: \Omega \to \Omega$ is a homeomorphism and that for all $T\in\Omega$ then the set $T+\rd$ is dense in $\Omega$. 

We now define a subspace of $\Omega$ which is important from the perspective of C*-algebras. We insist (without loss of generality) that each prototile contains the origin in its interior. If $T\in\Omega$ and $t\in T$, then $t = p+ x$ for some $p\in\p$ and $x\in\rd$ and the $p$ and $x$ are unique. We define the {\em puncture} of the tile $t$ to be $x$, and denote this point as $\x(t)$. If $P$ is a partial tiling, then we let $\x(P) = \{\x(t)\mid t\in p\}$ denote the set of all punctures of tiles in $P$. We let
\[
\op = \{ T\in \Omega \mid 0\in \x(T)\}.
\]
Then $\op$ is the set of all tilings in $\Omega$ which contain a tile whose puncture is the origin. This space is called the {\em punctured hull} or {\em transversal}. Given Assumptions \ref{primitive}--\ref{injectivesub}, the space $\op$ is compact, totally disconnected, and has no isolated points, and hence is homeomorphic to the Cantor set (see for example \cite{KP00}, p. 187). For a patch $P$ in some tiling in $\Omega$ and $t\in P$ the set
\[
U(P,t) = \{ T\in\op\mid P - \x(t)\subset T\}
\]
is clopen in $\op$. We note that for $x\in \rd$ the sets $U(P,t)$ and $U(P+x,t+x)$ are identical. As $P$ and $t$ vary, the sets $U(P,t)$ form a clopen basis for the topology on $\op$.

This paper concerns finite symmetries on tilings, and so we now define what symmetries we will consider.

\begin{defn}Let $(\p,\omega)$ be a substitution tiling system. We say that a group $G$ is a {\em symmetry group} for $(\p,\omega)$ if 
\begin{enumerate}
\item $G$ is a subgroup of $O(d,\re)$, the orthogonal group on $\rd$,
\item If $p\in\p$ and $g\in G$, then $gp = \{gx\mid x\in p\}$ is an element of $\p$ (ie, $G$ acts on $\p$ from the left), and
\item If $p\in\p$, then $\omega(gp) = g\omega(p)$. 
\end{enumerate}
If $G$ is such a group, then we say that $\mathcal{S}_G\subset\mathcal{P}$ is a set of {\em standard position} prototiles for $G$ if $G\mathcal{S}_G = \mathcal{P}$ and $\mathcal{S}_G$ does not properly contain any other such set.
\end{defn}
In this work, we will focus on the case where $G$ is a finite group.

\begin{example}\label{penrosesubex}
Figure \ref{penrosesub} below illustrates a substitution on a set of prototiles 
$$\ppen = \{\textbf{1}, \textbf{2}, \dots, \textbf{40}\}.$$ 
This is the substitution given in \cite{AP98}, Example 10.4. Only four prototiles are shown on the left; tiles {\bf 1} and {\bf 11} are congruent but are given different labels and substituted differently -- similarly for tiles {\bf 21} and {\bf 31}. If we let $r$ denote the counterclockwise rotation of $\re^2$ by $^\pi\hspace{-0.05cm}/_5$ and $f$ be the reflection over the $y$-axis, then the other tiles are given by $\textbf{2} = r\textbf{1}$, $\textbf{12} = r\textbf{11}$, and so on. We also have that $f\textbf{1} = \textbf{11}$ and $f\textbf{21} = \textbf{31}$. It is easy to check that this substitution is primitive and has finite local complexity. 

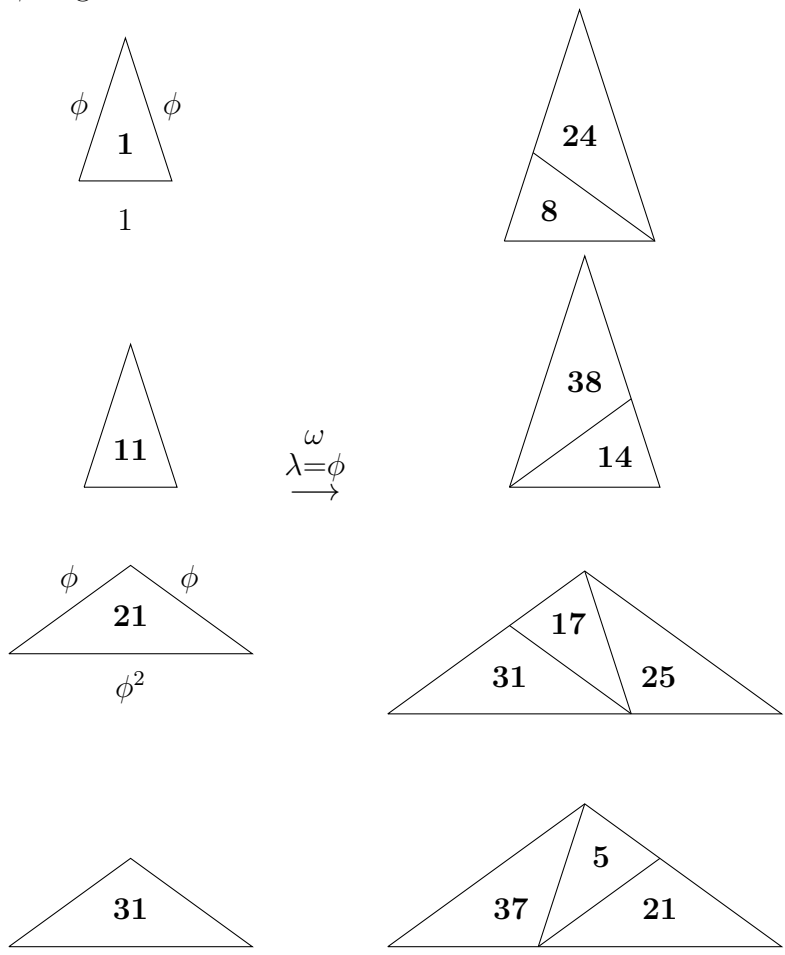
\begin{figure}
\begin{center}
\begin{tabular}{ccc}
Prototiles & & \\
(+ rotates by $^\pi\hspace{-0.05cm}/_5$) & & \\
$\phi =$ golden ratio & & \\
\begin{tikzpicture}[scale=2]
\draw (0,0.5) node {$\phi$};
\draw (0.618,0.5) node {$\phi$};
\draw (0.309,-0.4) node[anchor=south] {1};
\draw (0.309,0.25) node {\textbf{1}};
\draw (0,0) -- (0.618,0) -- (0.309,0.9510)-- (0,0);
\end{tikzpicture}
& & 
\begin{tikzpicture}[scale=2]
\draw (0,0) -- (1,0) -- (0.5,1.5388)-- (0,0);
\draw (1,0) -- (0.191, 0.588);
\draw (0.3, .2) node {\textbf{8}};
\draw (0.5, 0.7) node {\textbf{24}};
\end{tikzpicture}
\\
\vspace{0.2cm}
\begin{tikzpicture}[scale=2]
\draw (0,0) -- (0.618,0) -- (0.309,0.9510)-- (0,0);
\draw (0.309,0.25) node {\textbf{11}};
\end{tikzpicture}
& \Large $\stackrel{\omega}{\substack{\lambda = \phi\\ \longrightarrow}}$  &
\vspace{0.5cm}
\begin{tikzpicture}[scale=2]
\draw (0,0) -- (1,0) -- (0.5,1.5388)-- (0,0);
\draw (0,0) -- (0.809, 0.588);
\draw (0.7, .2) node {\textbf{14}};
\draw (0.5, 0.7) node {\textbf{38}};
\end{tikzpicture}
\\
\vspace{0.5cm}
\begin{tikzpicture}[scale=2]
\draw (0.4, 0.5) node {$\phi$};
\draw (1.2, 0.5) node {$\phi$};
\draw (0.809, -0.4) node[anchor=south] {$\phi^2$};
\draw (0,0) -- (1.618,0) -- (0.809,0.5877)-- (0,0);
\draw (0.809,0.25) node {\textbf{21}};
\end{tikzpicture}
& & 
\vspace{0.5cm}
\begin{tikzpicture}[scale=2]
\draw (0,0) -- (2.618,0) -- (1.309,0.9501)-- (0,0);
\draw (1.618,0) -- (0.809,0.5877);
\draw (1.618,0) -- (1.309,0.9501);
\draw (0.809,0.25) node {\textbf{31}};
\draw (1.8,0.25) node {\textbf{25}};
\draw (1.2,0.6) node {\textbf{17}};
\end{tikzpicture}
\\
\vspace{0.2cm}
\begin{tikzpicture}[scale=2]
\draw (0,0) -- (1.618,0) -- (0.809,0.5877)-- (0,0);
\draw (0.809,0.25) node {\textbf{31}};
\end{tikzpicture}
& &
\vspace{0.5cm}
\begin{tikzpicture}[scale=2]
\draw (0,0) -- (2.618,0) -- (1.309,0.9501)-- (0,0);
\draw (1,0) -- (1.809,0.5877);
\draw (1,0) -- (1.309,0.9501);
\draw (1.809,0.25) node {\textbf{21}};
\draw (0.818,0.25) node {\textbf{37}};
\draw (1.418,0.6) node {\textbf{5}};
\end{tikzpicture}
\end{tabular}
\caption{The Penrose substitution}\label{penrosesub}
\end{center}
\end{figure}
In the case of the Penrose tiling above, we can take $G$ to be the dihedral group $D_{10}$ generated by $r$ (the counterclockwise rotation by $^\pi\hspace{-0.05cm}/_5$) and $f$ (the reflection over the $x$-axis). These elements satisfy the relations
\[
r^{10} = f^2 = e, \hspace{1cm} frf = r^{-1}.
\]
In this case, we can take $\mathcal S_{D_{10}} = \{ \mathbf{1}, \mathbf{21}\}$. Another feature of this action is that $D_{10}$ acts freely on $\ppen$, that is, if $gp = p$ for some $g\in D_{10}$ and $p\in\ppen$, then $g = e$. We note that for the subgroup $\langle r\rangle$ we have $\mathcal S_{\langle r\rangle} = \{ \mathbf{1}, \mathbf{11}, \mathbf{21}, \mathbf{31}\}$ and the action of $\langle r\rangle$ also free.
\end{example}

If $G$ is a symmetry group for $(\p,\omega)$ and $t= p+x$ for $p\in\p$ and $x\in \rd$ then $gt = gp+gx$ is a tile. It is also clear that if $T$ is a (partial) tiling then $gT = \{gt\mid t\in T\}$ is also a (partial) tiling. 

\begin{proposition}The map $T\mapsto gT$ is a homeomorphism of $\Omega$, and so $\Omega$ is a $G$-space. Since $G$ acts on the prototiles, $\op$ is also a $G$-space.
\end{proposition}
\begin{proof}
It is straightforward to check that for $g\in G$ and $T_1, T_2\in\Omega$, we have that $d(gT_1, gT_2) = d(T_1, T_2)$.
\end{proof}

\section{Groupoids and C*-algebras}
In this section we set terminology and notation for topological groupoids and their C*-algebras. We also prove a result we need (Proposition \ref{semidirectcrossedproduct}) which seems to be well-known but for which we cannot locate a reference.

We will assume the theory of topological groupoids from \cite{R80}, with the exception that for a groupoid $\g$ the range and source maps $\g\to\g^{(0)}$ are denoted $r$ and $s$ respectively. We also use the word {\em subgroupoid} to mean a subspace of $\g$ which is closed under the groupoid operations and which has the same unit space as $\g$. A groupoid $\g$ is called {\em \'etale} if it is locally compact, Hausdorff, second countable, and $r$ and $s$ are local homeomorphisms. We note that this implies that $\g^{(0)}$ is open in $\g$. For $x,y\in\g^{(0)}$ denote
\[
\g^{x} = r^{-1}(x), \hspace{0.5cm}\g_{x} = s^{-1}(x), \hspace{0.5cm}\g_y^x = \g^x\cap\g_y,
\] 
and recall that $\g_x$ and $\g^x$ are discrete subspaces of $\g$. For $x\in\g^{(0)}$, $\g^x_x$ is a group, called the {\em isotropy group} at $x$. The groupoid $\g$ is called {\em principal} if for each $x,y\in\g^{(0)}$, there is at most one $\gamma \in \g$ such that $r(\gamma) = x$ and $s(\gamma) = y$. For $x\in\g^{(0)}$, the {\em orbit} of $x$ is the set $r(s^{-1}(x)) = s(r^{-1}(x))\subset \g^{(0)}$. A set $S\subset \g$ is a {\em graph} in $\g$ if the restrictions of $r$ and $s$ to $S$ are injective. Equivalently, $S$ is a graph if and only if $SS^{-1}$ and $S^{-1}S$ are both subsets of $\g^{(0)}$. For $E\subset \g^{(0)}$, we say that $E$ is {\em invariant} if $g\in\g$ and $s(g)\in E$ implies that $r(g)\in E$.

An \'etale groupoid $\g$ is called a {\em Cantor groupoid} if $\g^{(0)}$ is homeomorphic to the Cantor set. This is a definition of Phillips (\cite{Ph05}, Definition 1.1) and describes many of the groupoids associated to tilings.

An automorphism of a topological groupoid $\g$ is a self-homeomorphism which respects the groupoid structure on $\g$, and we denote the set of all automorphisms on $\g$ by Aut$(\g)$. Let $G$ be a group, and let $\alpha : G\to \textnormal{Aut}(\g)$ be a homomorphism. We write $$\gamma\cdot g = \alpha_{g^{-1}}(\gamma)$$ for $g\in G$ and $\gamma\in \g$ and note that this defines a right action of $G$ on $\g$. Renault (\cite{R80}, Definition I.1.7) defines the {\em semidirect product} $\g \rtimes_\alpha G$ as the groupoid $\g \times G$ where 
\begin{enumerate}
\item $(\gamma,g)$ and $(\xi,h)$ are composable if and only if $\xi = \eta\cdot g$ with $\gamma$ and $\eta$ composable,
\item $(\gamma,g)(\eta\cdot g, h) = (\gamma\eta,gh)$, and
\item $(\gamma,g)^{-1} = (\gamma^{-1}\cdot g, g^{-1})$.
\end{enumerate}
In this case, $r(\gamma,g) = (r(\gamma),e)$ and $s(\gamma,g) = (s(\gamma)\cdot g, e)$. In light of this, the unit space of $\g \rtimes_\alpha G$ may be identified with the unit space of $\g$. We may also drop the action $\alpha$ and simply write $\g\rtimes G$.

We will be concerned with the semidirect product of $r$-discrete groupoids by finite groups. The following proposition is slightly more general.

\begin{proposition}\label{semidirectrdiscrete}
Let $\g$ be an \'etale groupoid, let $G$ be a countable discrete group, and let $\alpha: G\to \textnormal{Aut}(\g)$ be a homomorphism. Then the semidirect product $\g\rtimes_\alpha G$ when given the product topology from $\g\times G$ is \'etale. If $\g$ is a Cantor groupoid, then so is $\g\rtimes G$.
\end{proposition}
\begin{proof}
We recall that the unit space of $\g\rtimes_\alpha G$ is also $\g^{(0)}$ and that $r(\gamma, g) = r(\gamma)$ for all $\gamma\in\g$ and $g\in G$. The groupoid $\g\rtimes_\alpha G$ is given the product topology, so it is locally compact, Hausdorff and second countable. Let $(\gamma, g)\in \g\rtimes_\alpha G$, and find a neighbourhood $U$ of $\gamma$ in $\g$ such that $r|_U: U\to r(U)$ is a homeomorphism. Then $U\times \{g\}$ is open in $\g\rtimes_\alpha G$ and $r(U\times \{g\}) = r(U)$, and so $r$ is a local homeomorphism.
\end{proof}
A Cantor groupoid $\g$ is called {\em approximately finite} (AF for short), if it is an increasing union of a sequence of compact open principal Cantor subgroupoids, each of which contains the unit space $\g^{(0)}$. A groupoid which is AF in this sense is AF in the sense of Renault (\cite{R80}, Definition III.1.1), and a groupoid is AF in the sense of Renault if and only if it is AF in the above sense and its unit space is compact with no isolated points (\cite{Ph05}, Proposition 1.16).

\begin{defn}\label{thin}(\cite{Ph05}, Definition 2.1)
Let $\g$ be a Cantor groupoid and let $K\subset \g^{(0)}$ be a compact subset. Then $K$ is called {\em thin} for $\g$ if for every $n$ there exist compact graphs $S_1, S_2, \dots, S_n\subset \g$ such that $s(S_k) = K$ and the sets $r(S_1), r(S_2), \dots, r(S_n)$ are pairwise disjoint.
\end{defn}
Before stating the following definition of Phillips, we recall that a measure $\mu$ on $\g^{(0)}$ is called {\em $\g$-invariant} if $\mu(r(S)) = \mu(s(S))$ for every open graph $S$.
\begin{defn}\label{almostaf}(\cite{Ph05}, Definition 2.2)
Let $\g$ be a Cantor groupoid. Then $\g$ is called an {\em almost AF Cantor groupoid} if we have the following:
\begin{enumerate}
\item There exists an open AF subgroupoid $\g_0\subset \g$ which contains the unit space such that whenever $K$ is a compact subset of $\g \setminus \g_0$, we have that $s(K)$ is thin for $\g_0$ in the sense of Definition \ref{thin}.
\item For every closed invariant subset $E\subset\g^{(0)}$, and every nonempty relatively open subset $U\subset E$, there is a $\g$-invariant Borel probability measure $\mu$ on $\g^{(0)}$ such that $\mu(U)>0$
\end{enumerate}
\end{defn}
A locally compact Hausdorff groupoid $\g$ is {\em essentially principal} if for every invariant closed subset $F$ of its unit space, the set of $x\in F$ for which $\g^x_x =\{x\}$ is dense in $F$. It is a fact that almost AF Cantor groupoids are essentially principal (\cite{Ph05}, Lemma 2.6).

Below we will associate Cantor groupoids to tilings, and the C*-algebras of a tiling will be the C*-algebras associated to these groupoids. By now the construction of a C*-algebra from an \'etale groupoid is quite well-known, but we will briefly describe it here.

Let $\g$ be an \'etale groupoid and consider the linear space $C_c(\g)$, the continuous compactly-supported functions from $\g$ to $\mathbb{C}$. For $f,g\in C_c(\g)$, the formulas
\[
fg (\gamma) := \sum_{\eta\in\mathscr G\atop r(\eta)=s(\gamma)} f(\gamma\eta)g(\eta^{-1}) \hspace{1cm} f^*(\gamma) := \overline{f(\gamma^{-1})}
\]
define a product and involution on $C_c(\g)$. The C*-algebra $C^*(\g)$ is the completion of this $*$-algebra in a suitable norm. We will work with the reduced C*-algebra. For $x\in\g^{(0)}$, there is a representation $\lambda_x$ of $C_c(\g)$ on $\ell^2(\g_x)$ given by
\[
\lambda_x(f)\xi(\gamma) = \sum_{\eta\in\mathscr G \atop s(\eta) = x}f(\gamma\eta^{-1})\xi(\eta).
\]
Then
\[
\|f\|_{\textnormal{red}} := \sup_{x}\{\|\lambda_x(f)\|\}
\]
defines a norm on $C_c(\g)$ which satisfies the C*-condition. The completion of $C_c(\g)$ under this norm is denoted $C^*_r(\g)$. For $f\in C_c(\g)$, let
\begin{equation}\label{rsnorms}
\|f\|_r = \sup_{u\in\mathscr G^{(0)}} \left\{\sum_{r(\gamma) = u}|f(\gamma)|\right\},\hspace{0.5cm} \|f\|_s = \sup_{u\in\mathscr G^{(0)}} \left\{\sum_{s(\gamma) = u}|f(\gamma)|\right\},
\end{equation}
\begin{equation}\label{Inorm}
\|f\|_I = \max\{\|f\|_r, \|f\|_s\}.
\end{equation}
If $f\in C_c(\mathscr G)$, then 
\[
\|f\|_\infty \leq \|f\|_{\textnormal{red}} \leq \|f\|_I.
\]
If $\g$ is a Cantor groupoid and $U\subset \g^{(0)}$ is clopen then $\chi_U$, the characteristic function on $U$, is a projection in $C_c(\g)$ satisfying, for $\gamma\in\g$ and $f\in C_c(\g)$,
\begin{equation}
(f\chi_U)(\gamma) = \begin{cases}f(\gamma)& s(\gamma)\in U\\ 0 & s(\gamma)\notin U\end{cases} \hspace{0.3cm}\text{and}\hspace{0.3cm} (\chi_Uf)(\gamma) = \begin{cases}f(\gamma)& r(\gamma)\in U\\ 0 & r(\gamma)\notin U\end{cases}.
\end{equation}

In defining Cantor groupoids in \cite{Ph05}, Phillips notes the following: a Cantor groupoid $\g$ is an almost AF Cantor groupoid if it satisfies Condition 1 of Definition \ref{almostaf} and either $C^*_r(\g)$ or $C^*_r(\g_0)$ is simple. By \cite{R80}, Proposition II.4.6, if $\g$ is an essentially principal \'etale groupoid then $C^*_r(\mathscr G)$ is simple if the only open invariant subsets of $\mathscr G^{(0)}$ are $\mathscr G^{(0)}$ and the empty set.

In the following theorem we record the results of Phillips \cite{Ph05} concerning almost AF Cantor groupoids.
\begin{theorem}\label{phillipslist}
Let $\g$ be an almost AF Cantor groupoid with respect to the AF groupoid $\g_0\subset \g$, and suppose that $C^*_r(\g)$ is simple. Then

\begin{enumerate}
\item $C^*_r(\g)$ has real rank zero,
\item $C^*_r(\g)$ has stable rank one,
\item The order on $K_0(C^*_r(\g))$ is determined by traces, and
\item The space of normalized traces on $C^*_r(\g_0)$ is in one-to-one correspondence with the space of normalized traces on $C^*_r(\g)$.
\end{enumerate}
\end{theorem}
\begin{proof}
See \cite{Ph05} Theorem 4.6, Theorem 5.2, Corollary 5.4, and Proposition 2.11.
\end{proof}

In our last result of this section, we prove that if $\g$ is an \'etale groupoid acted upon by a countable discrete group $G$, then $C^*_r(\g\rtimes G)$ is isomorphic to $C^*_r(\g)\rtimes_r G$, where the latter is the reduced crossed product. This result seems to be well-known, though we are unable to locate a reference and so include the proof for the sake of completeness. 

Recall (as in \cite{Dav}, for example) that if $A$ is a C*-algebra, $G$ is a countable discrete group and that $\alpha:G\to$ Aut$(A)$ is a homomorphism, then we may form the linear space $AG$ consisting of all finite linear combinations $\sum_{g\in G}a_g\delta_g$ with $a_g\in A$ which becomes a $*$-algebra when given product determined by the formal rule $\delta_ga\delta_{g^{-1}} = \alpha_g(a)$ and $\delta_{g}^* = \delta_{g^{-1}}$. A faithful representation $\rho$ of $A$ into $B(H_\rho)$ induces a faithful reprensentation $\tilde\rho$ of $A$ into $B(\ell^2(G, H_\rho))$ determined by $\tilde\rho(a)f(g) = \rho(\alpha_g^{-1}(a))(f(g))$. If $u$ is the usual left regular representation $u: G\to B(\ell^2(G, H_\rho))$, then there is a faithful representation $\tilde\rho\rtimes u$ of $AG$ into $B(\ell^2(G, H_\rho))$ given by
\[
\tilde\rho\rtimes u \left(\sum_{g\in G}a_g\delta_g\right) = \sum_{g\in G}\tilde\rho(a_g)u_g.
\]
The completion of $AG$ under the norm $\|a\| := \|\tilde\rho\rtimes u(a)\|$ is a C*-algebra norm independent of the faithful representation $\rho$. The reduced crossed product $A\rtimes_\alpha G$ is defined as the completion of $AG$ in this norm.
\begin{proposition}\label{semidirectcrossedproduct}
Let $\g$ be an \'etale groupoid, let $G$ be a countable discrete group, and let $\alpha: G\to \textnormal{Aut}(\g)$ be a homomorphism. Then 
\begin{enumerate}
\item $\alpha$ induces an action $\tilde\alpha: G \to$ Aut$(C^*_r(\g))$ such that for $f\in C_c(\g)$, $\gamma\in \g$ and $g\in G$ we have $\tilde\alpha_g(f)(\gamma) = f\left(\alpha_g^{-1}(\gamma)\right)$, and
\item there is a $*$-isomorphism $\Phi: C^*_r(\g)\rtimes_{\tilde\alpha, r} G\to C^*_r(\g\rtimes_\alpha G)$ such that for $f\in C_c(\g)$, $\gamma\in\g$, and $h,g\in G$ we have
\[
\Phi(f\delta_h)(\gamma,g) = \begin{cases}f(\gamma) & \text{if } g=h\\ 0 & \text{otherwise}.  \end{cases}
\]
\end{enumerate}
\end{proposition}
\begin{proof}
 
It is straightforward to verify that for $g\in G$, $\tilde\alpha_g$ is linear and that for $f_1, f_2\in C_c(\g)$ we have $\tilde\alpha_{g}(f_1f_2) = \tilde\alpha_{g}(f_1)\tilde\alpha_{g}(f_2)$ and $\tilde\alpha_g(f_1^*) = \tilde\alpha_g(f_1)^*$. The $*$-algebra $C_c(\g)$ inherits the norm from $C^*_r(\g)$ and it is straightforward to show that for each $g\in G$, $\tilde\alpha_{g}$ is continuous and so extends to a $*$-automorphism of $C^*_r(\g)$. Furthermore, for $g, h\in G$ and $f\in C_c(\g)$ we have that $\tilde\alpha_{gh}(f_1) = \tilde\alpha_g\circ\tilde\alpha_h(f_1)$, and so $\tilde\alpha$ is a homomorphism.
 
One shows that $\Phi: C_c(\g)G \to C_c(\g\rtimes_\alpha G)$ is an isomorphism of $*$-algebras. To show that $\Phi$ is continuous, let $x\in \g^{(0)} = \g\rtimes_\alpha G^{(0)}$ and recall that the reduced norm on $C^*_r(\g)$ is determined by the representations $\pi_x: C_c(\g) \to B(\ell^2(\g_x))$ given, for $f\in C_c(\g)$, $\xi \in \ell^2(\g_x)$ and $\gamma \in \g_x$ by 
\[
\pi_x(f)\xi(\gamma) = \sum_{\eta\in\g_x} f(\gamma\eta^{-1})\xi(\eta).
\]
The norm on $C^*_r(\g)\rtimes_{\alpha, r} G$ is induced by regular representations on $\ell^2(\g_x\times G)$ arising from the $\pi_x$. The representation $\pi_x$ induces a representation $\tilde\pi_x:C_c(\g) \to B(\ell^2(\g_x\times G))$ given, for $\xi \in \ell^2(\g_x\times G)$ and $(\gamma, g)\in \g_x\times G$ by 
\[
\tilde\pi_x\xi(\gamma,g) =  \sum_{\eta\in\g_x} \alpha_g^{-1}(f)(\gamma\eta^{-1})\xi(\eta, g).
\]
There is also a representation $u: G \to B(\ell^2(\g_x\times G))$ given by $(u_h\xi)(\gamma, g) = \xi(\gamma, h^{-1}g)$, and the norm of $C^*_r(\g)\rtimes_{\alpha, r} G$ is then determined by the representations $\tilde\pi_x \rtimes u: C_c(\g)G \to B(\ell^2(\g_x\times G))$ given, for $\sum_{h\in G} f_h\delta_h\in C_c(\g)G$ by
\[
\tilde\pi_x\rtimes u\left(\sum_{h\in G} f_h\delta_h\right) = \sum_{h\in G}\tilde\pi_x(f_h)u_h.
\]
The norm on $C^*_r(\g\rtimes_\alpha G)$ is determined by the representations $\lambda_x: C_c(\g\rtimes_\alpha G) \to B(\ell^2((\g\rtimes_\alpha G)_x))$ which is given, for $f\in C_c(\g\rtimes_\alpha G)$, $\xi \in \ell^2((\g\rtimes_\alpha G)_x)$ and $(\gamma, g)\in (\g\rtimes_\alpha G)_x$ by 
\[
\lambda_x(f)\xi(\gamma,g) = \sum_{(\eta,t)\in(\g\rtimes_\alpha G)_x} f((\gamma,g)(\eta,t)^{-1})\xi(\eta,t).
\]
There is an isomorphism of Hilbert spaces $\psi: \ell^2((\g\rtimes_\alpha G)_x) \to \ell^2(\g_x\times G)$ which is given, for $\xi \in \ell^2((\g\rtimes_\alpha G)_x)$ and $(\gamma, g)\in \g_x\times G$ by 
\[
\psi(\xi)(\gamma, g) = \xi(\gamma\cdot g^{-1}, g).
\]
This induces a $*$-isomorphism $\Psi: B(\ell^2(\g_x\times G)) \to B(\ell^2((\g\rtimes_\alpha G)_x))$ which is given, for $T\in B(\ell^2(\g_x\times G))$, $\xi\in \ell^2((\g\rtimes_\alpha G)_x)$ and $(\gamma, g)\in \g_x\times G$ by
\[
\Psi(T)(\xi)(\gamma, g) = \psi^{-1}\circ T \circ\psi(\xi)(\gamma, g) = T(\psi(\xi))(\gamma\cdot g, g).
\]
We claim that the diagram
\[
\xymatrix{
C_c(\g)G \ar[r]^{\tilde\pi_x\rtimes u} \ar[d]^{\Phi}&B(\ell^2(\g_x\times G))\ar[d]^{\Psi} \\
C_c(\g\rtimes_\alpha G) \ar[r]^{\lambda_x} & B(\ell^2((\g\rtimes_\alpha G)_x))
}
\]
commutes. 

To see this, first take $f\in C_c(\g)$, $h\in G$, $\xi\in \ell^2((\g\rtimes_\alpha G)_x)$ and $(\gamma, g)\in(\g\rtimes_\alpha G)_x$. We calculate
\[
\lambda_x\big(\Phi(f\delta_h)\big)\xi(\gamma,g) = \sum_{(\nu,t)\in(\g\rtimes_\alpha G)_x} \big(\Phi(f\delta_h)\big)\left((\gamma,g)(\nu,t)^{-1}\right)\xi(\nu,t).
\] 
Using the rules of the semidirect product, one calculates the product $(\gamma,g)(\nu,t)^{-1} = (\gamma(\nu^{-1}\cdot tg^{-1}), gt^{-1})$. A term $\big(\Phi(f\delta_h)\big)(\gamma(\nu^{-1}\cdot tg^{-1}), gt^{-1})$ is only nonzero if $gt^{-1} = h$, and in this case we have $t = h^{-1}g$ and $tg^{-1} = h^{-1}$. Hence we have 
\[
\lambda_x\big(\Phi(f\delta_h)\big)\xi(\gamma,g) = \sum_{(\nu,h^{-1}g)\in(\g\rtimes_\alpha G)_x} f\left(\gamma(\nu^{-1}\cdot h^{-1}\right)\xi(\nu,h^{-1}g).
\] 
We have that $(\nu,h^{-1}g)\in(\g\rtimes_\alpha G)_x$ if and only if $(\nu\cdot h^{-1}, g)\in(\g\rtimes_\alpha G)_x$, and so setting $\eta= \nu\cdot h^{-1}$ we have
\[
\lambda_x\big(\Phi(f\delta_h)\big)\xi(\gamma,g) = \sum_{(\eta,g)\in(\g\rtimes_\alpha G)_x} f\left(\gamma\eta^{-1}\right)\xi(\eta\cdot h,h^{-1}g).
\]
On the other hand, we have
\begin{eqnarray*}
\Psi\left(\tilde\pi_x\rtimes u(f\delta_h)\right)\xi(\gamma, g) &=& \Psi(\tilde\pi_x(f)u_h)\xi(\gamma,g)\\                              
                              &=&\tilde\pi_x(f)u_h(\psi(\xi))(\gamma\cdot g, g)\\
                              &=&\pi_x(\alpha_g^{-1}(f))(u_h(\psi(\xi)))(\gamma\cdot g, g)\\
                              &=&\sum_{\eta\cdot g\in\g_x} \alpha_g^{-1}(f)((\gamma\cdot g)(\eta^{-1}\cdot g))(u_h(\psi(\xi))(\eta\cdot g, g)\\      
                              &=&\sum_{\eta\cdot g\in\g_x} f(\gamma\eta^{-1})u_h(\psi(\xi))(\eta\cdot g, g)\\
                              &=&\sum_{\eta\cdot g\in\g_x} f(\gamma\eta^{-1})\psi(\xi)(\eta\cdot g, h^{-1}g)\\   
                              &=&\sum_{\eta\cdot g\in\g_x} f(\gamma\eta^{-1})\xi(\eta\cdot h, h^{-1}g)\\   
                              &=&\lambda_x\big(\Phi(f\delta_h)\big)\xi(\gamma,g).          
\end{eqnarray*}
Since $\left(\bigoplus_{x\in \g^{(0)}}\tilde\pi_x\right)\rtimes u$ is a faithful representation of  $C^*_r(\g)\rtimes_{\tilde\alpha, r} G$ and $\bigoplus_{x\in \g^{(0)}}\lambda_x$ is a faithful representation of $C^*_r(\g\rtimes_\alpha G)$, we have that $\Phi$ extends to a $*$-isomorphism of $C^*_r(\g)\rtimes_{\tilde\alpha, r} G$ and $C^*_r(\g\rtimes_\alpha G)$.
\end{proof}

\section{Tiling Groupoids and C*-algebras}\label{tilingalgebra}

In this section we summarize facts about groupoids and C*-algebras associated to tilings. Most of the items in this section are well-known (and for a good introductory reference, see \cite{KP00}); we include them here for completeness and later reference. 

Given a substitution tiling system $(\p,\omega)$, the equivalence relation
\[
\rp = \{ (T, T+x)\mid T, T+x\in\op, x\in\rd\},
\]
is an \'etale groupoid. Its unit space $\rp^{(0)} = \op$ is homeomorphic to the Cantor set and so $\rp$ is a Cantor groupoid. Let $P$ be a patch in some tiling in $\Omega$, let $t_1, t_2\in P$, and set
\begin{equation}\label{VPt1t2}
V(P,t_1, t_2) = \{(T,T')\in\op\mid P-\x(t_1)\subset T, P-\x(t_2)\subset T'\}.
\end{equation}
These sets are compact open graphs, and it is easily checked that $r(V(P,t_1,t_2)) = U(P,t_1)$ and $s(V(P,t_1,t_2)) = U(P, t_2)$. The collection of such sets generate the topology on $\rp$.

There is a natural AF subgroupoid of $\rp$. If $t$ is a tile and $n\in\NN$, then we call $\omega^n(t)$ an {\em $n$th-order supertile}. Invertibility of $\omega:\Omega\to\Omega$ implies that for each $n\in\NN$, every tiling $T\in\Omega$ has a unique decomposition into $n$th-order supertiles, and that these decompositions are nested. For each $n\in\NN$, define a subgroupoid $\mathcal{R}_n\subset \rp$ by saying that $(T, T-x)\in\mathcal{R}_n$ if 0 and $x$ are punctures in the same $n$th-order supertile in $T$'s unique decomposition into $n$th-order supertiles. 

The subgroupoids $\mathcal{R}_n$ also have a description in terms of compact open graphs. For $p\in\p$ and $n\in\NN$ let Punc$(n,p)$ be the set of punctures in $\omega^n(p)$. For $x,y\in$ Punc$(n,p)$, define
\[
E^n_p(x,y) = \left\{\left(\omega^n(T) - x, \omega^n(T) - y\right) \mid T\in U(\{p\}, p)\right\}.
\]
It is easy to check that these are compact open graphs, and that $\mathcal{R}_n$ is the disjoint union of $E^n_p(x,y)$ as $p$ ranges over $\p$ and $x,y$ range over Punc$(x,y)$. The union of this nested sequence of compact open subgroupoids is denoted $$\raf := \cup \mathcal{R}_n.$$ 

At this point, it is not clear that $\raf$ is not all of $\rp$, so we give an example where we do not have equality.

On the left in the below picture, there is a patch from a Penrose tiling, which we call $P$. On the right is $\omega^2(P)$. 
\begin{center}
\includegraphics[scale = 0.1]{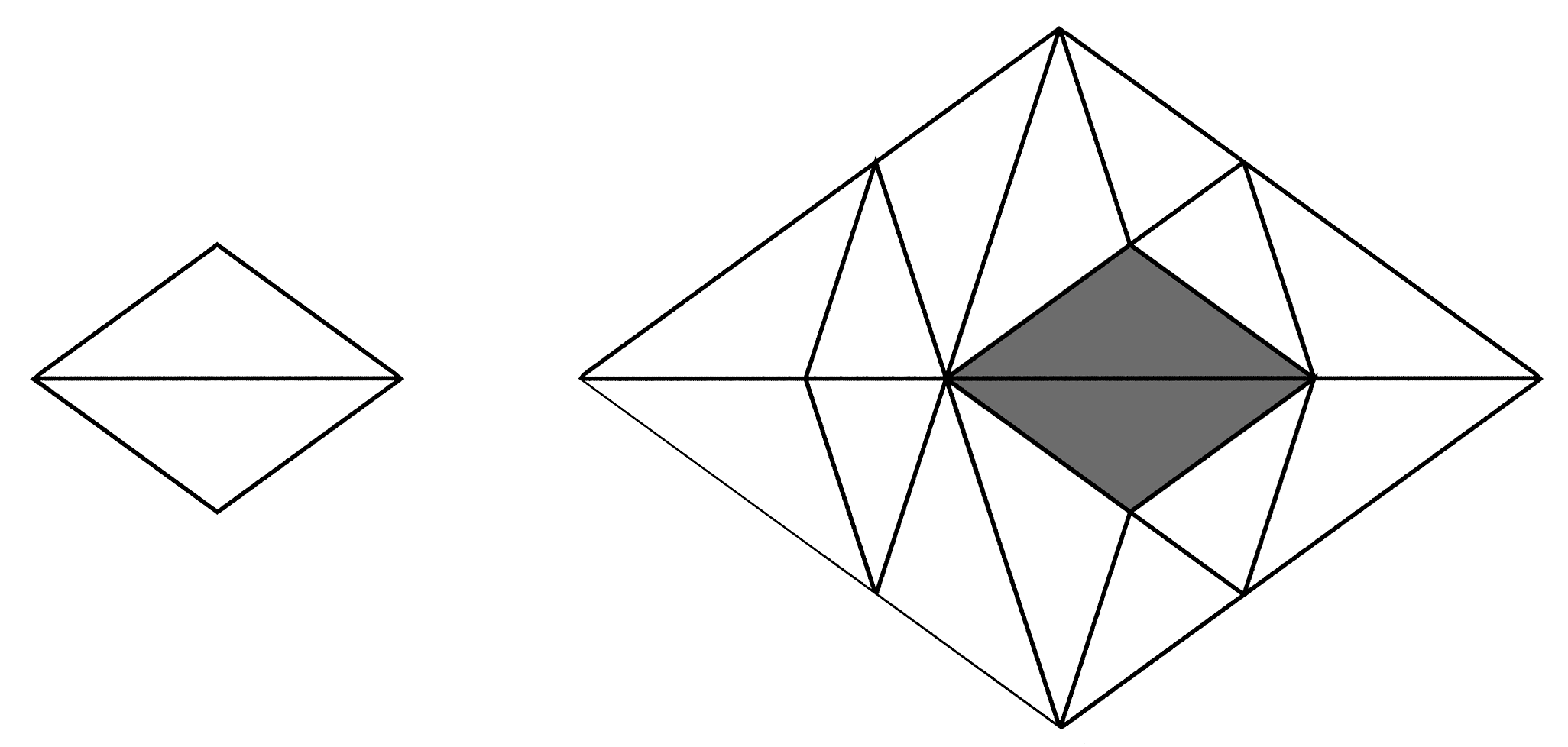}
\end{center}
We see that $P\subset \omega^2(P)$ (after choosing an origin for $P$ as we did in the discussion after Assumption \ref{primitive}), and so setting $T = \cup \omega^{2i}(P)$ yields a tiling which has a biinfinite line consisting of edges of tiles through the origin. If $x$ is a puncture of $T$ above this line and $y$ is a puncture below it, then $x$ and $y$ are never contained in the same $n$th order supertile for any $n$, and so $(T-x, T-y)\in\rp \setminus \raf$.
 
Relative to $\raf$, $\rp$ is an almost AF Cantor groupoid. The key observation needed to show this (made essentially by Putnam in \cite{Pu00} and stated in this form by Phillips \cite{Ph05}) is that, if for $r>0$ we set
\begin{equation}\label{Lr}
L_r = \{ (T, T-x)\in \rp  \setminus \raf \mid \|x\|\leq r\}
\end{equation}
and $K\subset \rp  \setminus \raf$ is a compact set, then $s(K)$ is contained in $r(L_r)$ for some $r$. Furthermore, $r(L_r)$ is thin, and so $s(K)$ is as well. We note that this depends on $(\p, \omega)$ forcing the border (Assumption \ref{forcesborder}) and the capacity of the boundary of each prototile being strictly less than $d$. The only invariant subsets of $\op$ are $\op$ or the empty set, and so $\rp$ is an almost AF Cantor groupoid (\cite{Ph05}, Theorem 7.1).

We denote $$\Aw := C^*_r(\rp)$$ $$AF_\omega := C^*_r(\raf).$$
These were denoted $A_T$ and $AF_T$ respectively in \cite{KP00}, and there it was suggested that $\Aw$ and $AF_\omega$ might be more appropriate. We adopt this view to emphasize the dependence on the substitution rather than any one particular tiling. This C*-algebra was defined by Kellendonk \cite{KelNCG} and studied further in \cite{KP00}, \cite{Pu00}, and later \cite{Ph05}. Since $\rp$ is an almost AF Cantor groupoid, $\Aw$ enjoys the properties listed in Theorem \ref{phillipslist}.

There is a convenient presentation of $AF_\omega$ as an inductive limit of finite dimensional C*-algebras which we will now briefly summarize. For the details, see \cite{KP00} pp. 199-201. For $x,y\in$ Punc$(x,y)$, let $e^n_p(x,y)$ denote the characteristic function of $E^n_p(x,y)$. These are elements of $AF_\omega\subset \Aw$.  Then for $p,p'\in\p$, $x, y\in$ Punc$(n,p)$, and $x', y'\in$ Punc$(n,p')$ we have
\[
\begin{array}{rll}
e^n_p(x,y)e^n_{p'}(x', y')&= 0 & \text{if }p\neq p'\\
e^n_p(x,y)e^n_{p'}(x', y')&= 0 & \text{if }p= p' \text{ and }y \neq x'\\
e^n_p(x,y)e^n_{p'}(x', y')&= e^n_p(x, y') & \text{if }p= p' \text{ and }y= x'.
\end{array}
\]
These imply that if we fix $n\in\NN$ and $p\in\p$ and let
\[
A_{n,p} = \textnormal{span}_{\mathbb{C}}\left\{e^n_p(x,y) \mid x,y\in \textnormal{ Punc}(n,p)\right\}
\]
then $A_{n,p}$ is a $*$-subalgebra of $\Aw$ isomorphic to the $(m\times m)-$matrices, where $m =$ \#Punc$(n,p)$. Furthermore, if $p\neq p'$, then $A_{n,p}$ and $A_{n,p'}$ are orthogonal, and hence their direct sum
\[
A_n := \bigoplus_{p\in\mathcal P} A_{n,p}
\]
is also a subalgebra of $\Aw$. It is straightforward to verify that for $n\in\NN$ we have $A_n\subset A_{n+1}$, and that the identity of $\Aw$ is $$\sum_{p\in\p}e^0_p(0,0)\in A_0,$$ and so the identity is in $A_n$ for all $n\in\NN$. For any $n\in \NN$, the identity can be written as
\begin{equation}\label{identity}
1=\sum_{p_i\in\mathcal{P}} \sum_{x\in\textnormal{Punc}(n,p_i)}e^n_{p_i}(x,x).
\end{equation}
It is a fact that $AF_\omega = \overline{\cup A_n}$.
 
The unital inclusion $A_n\hookrightarrow A_{n+1}$ has a nice description in terms of the substitution. Recall that if  $\varphi:A \to B$ is a unital $*$-homomorphism with $A = \oplus_{i=1}^k\mathbb M_{n_i}$ and $B = \oplus_{i=1}^l\mathbb M_{m_i}$ finite dimensional, then $\varphi$ is determined up to unitary equivalence in $B$ by an $l\times k$ matrix $M$ of nonnegative integers called the matrix of {\bf partial multiplicities}. If $M = [M_{ij}]$, then the integer $M_{ij}$ is the multiplicity of the embedding of the summand $\mathbb M_{n_j}$ of $A$ into the summand $\mathbb M_{m_i}$ of $B$. For details see \cite{Dav} Lemma III.2.1.

One can obtain the matrix of partial multiplicities is through traces. If $\tau$ is a trace on $\mathbb{M}_n$, then it is a positive scalar multiple of the usual matrix trace Tr (this is the sum of the diagonal entries). If $A = \oplus_{i=1}^k\mathbb M_{n_i}$, then for each $j$,
\[
\tau^A_j\left((a_i)_{i=1}^k\right) = \textnormal{Tr}(a_j)
\]
is a trace on $A$. Furthermore, every trace on $A$ can be written as a positive linear combination of the $\tau^A_j$ since restricting to a summand yields a trace on that summand. Let $B = \oplus_{i=1}^l\mathbb M_{m_i}$ and suppose that $\varphi: A\to B$ is a unital injective homomorphism of C*-algebras. Then for each $i$ between 1 and $l$, $\tau^B_i\circ\varphi$ is a trace on $A$. Furthermore, if we denote by $q_i$ the identity on the $i$th summand in $A$, $\tau^B_i\circ\varphi(q_s)$ should be the trace of $q_s$ multiplied by the multiplicity of the embedding of the summand $\mathbb M_{n_s}$ of $A$ into the summand $\mathbb M_{m_i}$ of $B$. On the other hand, we know that
\begin{equation}\label{incidencematrixtracesequation}
\tau^B_i\circ\varphi = \sum_{j=1}^{k}M_{ij}\tau^A_j
\end{equation}
for some positive scalars $M_{ij}$. Hence,
\[
\tau^B_i\circ\varphi(q_s) = \sum_{j=1}^{k}M_{ij}\tau^A_j(q_s) =  M_{is}\tau^A_s(q_s) =  M_{is}n_s,
\]
and so $ M = [ M_{ij}]$ is the matrix of partial multiplicities of the inclusion. A formula for its entries is given by manipulating the above,
\begin{equation}\label{incidenceformula}
 M_{ij} = \frac{\tau^B_i\circ\varphi(q_j)}{\tau^A_j(q_j)}.
\end{equation}

One can show that the matrix of partial multiplicities of the unital inclusion $A_n\hookrightarrow A_{n+1}$ is independent of $n$, and is the $(\Npro \times \Npro)$ matrix $M$ whose $(i,j)$th entry is the number of translates of prototile $p_j$ in $\omega(p_i)$ (see \cite{KP00}, Section 9). Since $\omega$ is primitive, the matrix $M$ is primitive in the sense that there exists $k\in\NN$ such that $M^k$ has strictly positive entries.

We now turn our attention to traces on $\Aw$. By Theorem \ref{phillipslist}, the traces on $\Aw$ are in one-to-one correspondence with the traces on $AF_\omega$. In our case, $AF_\omega$ is an AF algebra with a constant primitive matrix of partial multiplicites. By \cite{Ha81} Theorem 4.1, such an AF algebra has a unique tracial state. Hence $\Aw$ has a unique tracial state as well. This trace is given by integration against a unique $\rp$-invariant probability measure $\mu$ -- for details on this see \cite{KelNCG}, \cite{KP00} or \cite{Pu00}.  We describe the essential properties of this trace and how to calculate it on elements of $AF_\omega$. 

Since the matrix $M$ is primitive, by the Perron-Frobenius Theorem $M$ admits left and right eigenvectors whose entries are all positive and whose eigenvalue is positive and strictly larger in modulus than the other eigenvalues of $M$. For a primitive substitution tiling system $(\mathcal P, \omega)$ in $\rd$ with expansion constant $\lambda$, the Perron eigenvalue is $\lambda^d$. Furthermore, if $\mathcal{P} = \{p_1, p_2, \dots, p_{\Npro}\}$ and $v_R$ is the vector whose $i$th entry is the volume of $p_i$, then $v_R$ is a right Perron-Frobenius eigenvector of $M$, see \cite{So98}, Corollary 2.4 (note that the substitution matrix as defined by Solomyak is the transpose of our substitution matrix). If $v_L$ is the vector whose $i$th entry is the relative frequency of translates of the prototile $p_i$ in any tiling $T\in \Omega$, then $v_L$ is a left Perron-Frobenius eigenvector of  $M$, see \cite{KelNCG}, Section 4. 

Now, given a basis element $e^n_p(x,y)$, its trace is
\begin{equation}\label{traceonafw}
\tau(e^n_{p_i}(x,y)) = \left\{\begin{array}{ll} \lambda^{-dn}v_L(i)& \textnormal{if } x=y\\ 0 & \textnormal{if } x\neq y\end{array}\right. .
\end{equation}
Normalizing $v_L$ so that $\tau$ is a tracial state gives us
\[
\sum_{p_i\in \mathcal P}v_L(i) = 1
\]
and, applying $\tau$ to both sides of (\ref{identity}) yields
\begin{equation}
\sum_{p_i\in\mathcal{P}}\#\textnormal{Punc}(n,p_i)\lambda^{-dn}v_L(i)=1.\label{tracialstateformula}
\end{equation}
We will use this to prove results related to the weak Rokhlin property in Section \ref{RPsection}.

To close this section, we prove that $\Aw$ is finitely generated. What we prove here is certainly the same idea as \cite{KelCoin}, paragraph 4, but we state it in terms elements of $\Aw$. 

If we let $e(P, t_1, t_2)$ be the characteristic function of $V(P,t_1,t_2)$ (from (\ref{VPt1t2})), then $e(P,t_1,t_2)$ is an element of $C_c(\rp)$. Let $P, P'$ be patches and let $t_1, t_2, t\in P$ and $t_1', t_2'\in P'$. Assume without loss of generality that $x_{t_2} = 0$ and that $x_{t_1'} = 0$. Then we have the following.
\begin{enumerate}
\item The product $e(P, t_1, t_2)e(P',t_1', t_2')$ is nonzero precisely when $U(P, t_1)\cap U(P',t_2) \neq \emptyset$ and the patches $P$ and $P'$ agree on the overlap of their supports, i.e., $P \cup P'$ is a patch. In this case the product is $e(P\cup P', t_1, t_2')$.
\item $e(P, t_1, t_2)^* = e(P, t_2, t_1)$.
\item $e(P, t, t) e(P, t, t) = e(P, t, t)$.
Hence each $e(P, t, t)$ is a projection and $e(P, t_1, t_2)$ is a partial isometry from $e(P, t_2, t_2)$ to $e(P, t_1, t_1)$ in $C_c(\rp)$.
\end{enumerate}
The linear span of the set
\[
\mathcal{E} = \{e(P,t_1, t_2) \mid P \text{ is a patch in some }T\in\Omega; t_1, t_2\in P\}
\]
is dense in $C_c(\rp)$ and hence in $\Aw$ (see \cite{KP00}, Section 4). Suppose that $t_1$ and $t_2$ are tiles in some tiling in $\Omega$ and that Int$(t_1\cup t_2)$ is connected, that is, $t_1$ and $t_2$ are adjacent. We write
\[
e_{t_1t_2} = e(\{t_1, t_2\},t_1, t_2)
\]
and set
\[
\mathcal{E}_2 = \{e_{t_1t_2}\mid \text{Int}(t_1\cup t_2)\text{ is connected}\}.
\]
Then $\mathcal{E}_2$ is finite by finite local complexity. We will show that every element of $\mathcal{E}$ can be written as a finite sum of a finite product of elements of $\mathcal{E}_2$; this will show that $\mathcal{E}_2$ is a generating set for $\Aw$.

\begin{lemma}\label{finitesumconnected}
Let $P$ be a patch with $t_1, t_2\in P$ and $x_{t_1} = 0$. Let $r>0$ be such that supp$(P)\subset B_r(0)$ and let
\[
Y = \{ T(B_r(0))\mid T\in U(P, t_1)\}.
\]
Then $Y$ is a finite set and
\[
V(P, t_1, t_2) = \dot{\bigcup_{P'\in Y}}V(P', t_1, t_2)
\]
where the union is disjoint.
\end{lemma}
\begin{proof}
That $Y$ is finite follows from finite local complexity. Take $P_1, P_2\in Y$ with $P_1 \neq P_2$, and suppose that $(T,T+x)\in V(P_1, t_1, t_2)\cap V(P_2, t_1, t_2)$. This implies that $T\in U(P_1, t_1)\cap U(P_2, t_1)$, and hence $P_1, P_2\subset T$. But this means that 
\[
P_1 = P_1(B_r(0)) = T(B_r(0)) = P_2(B_r(0)) = P_2,
\]
a contradiction, and hence the sets $V(P', t_1, t_2)$ are pairwise disjoint.

Let $(T, T+x)\in V(P', t_1, t_2)$ for some $P'\in Y$. Then since $P\subset P'$, we must have that $(T, T+x)\in V(P, t_1, t_2)$. Conversely suppose that $(T, T+x)\in V(P, t_1, t_2)$. Then define $P' = T(B_r(0))$. We have $P'\in Y$ and so $(T,T+x)\in V(P', t_1, t_2)$.
\end{proof}
If we let $\mathcal{E}_c = \{e(P,t_1, t_2)\in \mathcal{E}\mid P\text{ is a connected patch}\}$, then Lemma \ref{finitesumconnected} tells us that the span of $\mathcal{E}_c$ is dense in $\Aw$ as well. We now show that each element of $\mathcal{E}_c$ can be written as a product of elements of $\mathcal{E}_2$.

\begin{lemma}\label{finiteproduct}
If $e(P, t_1, t_2)\in \mathcal{E}_c$, it is a finite product of elements of $\mathcal{E}_2$.
\end{lemma}
\begin{proof}
Let $P = \{t_1, t_2, t_3, \dots, t_n\}$. Assume without loss of generality that $x_{t_1} = 0$. 
For each $1\leq i\leq n$, there exist tiles $s_1, s_2, \dots, s_{k_i}$ in $P$ such that $s_1 = t_1$, $s_{k_i} = t_i$, and for all $1\leq j\leq k_i$ we have $\{s_j, s_{j+1}\}$ is a connected patch. Let
\[
w_i := e_{s_1s_2}e_{s_2s_3}\cdots e_{s_{k_i-1}s_{k_i}}.
\]
Then we see that
\[
w_i = e\left(\bigcup_{m=1}^{k_i}\{s_m\}, t_1, t_i\right)
\]
and
\[
w_iw_i^* = e\left(\bigcup_{m=1}^{k_i}\{s_m\}, t_1, t_1\right).
\]
Finally, if we take the product of all of these, we see that the patch obtained must contain each tile in $P$, so that
\[
\prod_{i=1}^{n}w_iw_i^* = e(P, t_1, t_1),
\]
and so
\[
\left(\prod_{i=1}^{n}w_iw_i^*\right)w_2^* = e(P, t_1, t_1)e\left(\bigcup_{m=1}^{k_2-1}\{t_{j_m}\}, t_1, t_2\right) = e(P, t_1, t_2).
\]
\end{proof}
\begin{proposition}\label{Awfinitelygenerated}
The finite set $\mathcal{E}_2$ is a generating set of $\Aw$.
\end{proposition}
\begin{proof}
This follows from Lemmas \ref{finitesumconnected} and \ref{finiteproduct} along with \cite{KP00}, Section 4.
\end{proof}

\section{Symmetry Group Actions on Tiling Groupoids and C*-algebras}\label{crossedproductsection}

 In this section we show that if $G$ is a symmetry group for $(\p,\omega)$ which acts freely on $\p$, then $\rp\rtimes G$ is an almost AF Cantor groupoid with respect to $\raf\rtimes G$, which we show is an AF groupoid. We will then conclude from Proposition \ref{semidirectcrossedproduct} that $\Aw\rtimes G$ enjoys the properties listed in Theorem \ref{phillipslist}. We also show that the reduced C*-algebra of $\raf\rtimes G$ has a unique trace, and conclude that $A_\omega\rtimes G$ also has a unique trace.

If $G$ is a symmetry group for $(\p,\omega)$, then $G$ acts on $\rp$ and $\raf$. The action $\alpha: G \to$ Aut$(\rp)$ is given for $(T, T')\in\rp$ by
\[
\alpha_g(T,T') = (gT, gT').
\] 

An element of $\rp\rtimes_\alpha G$ is of the form $((T, T'), g)$ for $(T, T')\in\rp$ and $g\in G$. If $\gamma = ((T,  T'), g)$, then 
\[
\gamma^{-1} = ((g^{-1}T', g^{-1}T), g^{-1}), \hspace{0.5cm} s(\gamma) = g^{-1}T', \hspace{0.5cm} r(\gamma) = T,
\]
From now on we omit the subscript $\alpha$ and write $\rp\rtimes G$ since we have only one action to consider.

\begin{lemma}\label{semidirectofafisaf}
Suppose that $G$ is a finite symmetry group for $(\mathcal{P},\omega)$ and that $G$ acts freely on $\mathcal{P}$. Let $\raf$ be the AF Cantor groupoid associated to $(\mathcal{P},\omega)$. Then $\raf \rtimes G$ is an AF Cantor groupoid.
\end{lemma}
\begin{proof}
It is enough to show that $\raf \rtimes G$ is an increasing union of compact open principal subgroupoids each with unit space $\op$. For $n\in\NN$, we first show that $\mathcal R_n\rtimes G$ is principal. Let $\gamma_i =((T_i, T_i + x_i),  g_i)$, for $i=1,2$, be elements of $\mathcal{R}_n\rtimes G$ and suppose that $r(\gamma_1) = r(\gamma_2)$ and $s(\gamma_1) = s(\gamma_2)$. For $i = 1,2$, we have $r(\gamma_i)= T_i$ and so $T_1 = T_2:=T$. This gives us $\gamma_1 =((T, T + x_1),  g_1)$ and $\gamma_2 = ((T, T + x_2),  g_2)$. For $i=1,2$ we have $s(\gamma_i) = g_i^{-1}(T+x_i)$ and so $ g_1^{-1}(T+x_1) =  g_2^{-1}(T+x_2)$, or $T+x_1 =  g_1 g_2^{-1}(T+x_2)$. The pairs $(T, T+x_1)$ and $(T,T+x_2)$ are both in $\mathcal R_n$. This means that $\omega^{-n}(T+x_1)$ and $\omega^{-n}(T+x_2)$ are both tilings with the same tile around the origin, only translated. That is to say that $\omega^{n}(T+x_1)(0) = t$ and $\omega^{-n}(T+x_2)(0) = t + \lambda^{-n}(x_1-x_2)$. But the above then implies that $t =  g_1 g_2^{-1}(t + \lambda^{-n}(x_1-x_2))$. 
There exists unique $p\in \p$ and $y\in\rd$ such that $t= p+y$, and so this implies that
\[
p = g_1g_2^{-1}p + g_1g_2^{-1}x + \lambda^{-n}(x_1-x_2) - x.
\]
Since $p$ and $g_1g_2^{-1}p$ are both prototiles and one is a translate of the other we must have that $p = g_1g_2^{-1}p$ and $g_1g_2^{-1}x + \lambda^{-n}(x_1-x_2) - x=0$. Since $G$ acts freely on $\p$ we have that $g_1=g_2$ and so $x_1 = x_2$. Thus each $\mathcal{R}_n\rtimes G$ is principal. It is easy to see that 
\[
\raf \rtimes  G = \bigcup_{n\in\NN}\mathcal R_n\rtimes G
\] 
and so $\raf \rtimes  G$ is an increasing union of compact principal groupoids. Since $\mathcal R_n\rtimes G$ inherits the product topology from $\mathcal R_n\times G$ and $\mathcal R_n$ is open in $\mathcal R_{n+1}$, we must have that $\mathcal R_n\rtimes G$ is open in $\mathcal R_{n+1}\rtimes G$. This completes the proof.
\end{proof}

\begin{lemma}\label{semidirectessentiallyprincipal}
Suppose that $G$ is a finite symmetry group for $(\mathcal{P},\omega)$, and suppose that $G$ acts freely on $\p$. Then the groupoid $\rp\rtimes G$ is essentially principal. 
\end{lemma}
\begin{proof}
Find, as we did in the discussion after Assumption \ref{primitive}, $T\in\Omega$ such that $\omega^n(T) = T$
and $T$ is the nested union of $kn$-th order supertiles. The set $T - \x(T)$ is the intersection of the translational orbit of $T$ with $\op$, and is dense in $\op$. We take $T'\in T-\x(T)$ and show that it has trivial isotropy group. To do this, we take an element $((T', T'-x), g)$ whose range is $T'$, assume that its source is also $T'$ and show that $g=e$ and $x=0$. Since $T$ is the nested union of $kn$-th order supertiles $T'$ is as well, so it is possible to find $k$ such that $0$ and $x$ are punctures in the same $kn$-th order supertile $\omega^{kn}(t)$. 
Thus both $((T', T'-x), g)$ and $((T',T'), e)$ are elements of $\mathcal{R}_{kn}\rtimes G$, which is a principal groupoid by the proof of Lemma \ref{semidirectofafisaf}. Since these two elements have the same range and source, they must be equal. Hence $T'$ has trivial isotropy group and so $\rp \rtimes G$ is essentially principal.
\end{proof}
\begin{lemma}\label{minimalsemidirectgroupoid} Suppose that $G$ is a finite symmetry group for $(\mathcal{P},\omega)$, and suppose that $G$ acts freely on $\p$. Then the only open invariant subsets of $\op$ with respect to $\rp\rtimes G$ are $\emptyset$ and $\op$. Hence, $C^*_r(\rp\rtimes G)$ is simple.
\end{lemma}
\begin{proof}
Every $\rp\rtimes G$-orbit in $\op$ is the union of $\rp$-orbits, and each $\rp$-orbit is dense in $\op$. Hence every $\rp\rtimes G$-orbit is dense, and so the only nonempty closed invariant subset of $\op$ is $\op$. By Lemma \ref{semidirectessentiallyprincipal}, $\rp\rtimes G$ is essentially principal, and so by \cite{R80} Proposition II.4.6, $C^*_r(\rp\rtimes G)$ is simple.
\end{proof}
We are now in a position to prove the following theorem.
\begin{theorem}\label{RpxGalmostAF}
Suppose that $G$ is a finite symmetry group for $(\mathcal{P},\omega)$ and that $G$ acts freely on $\mathcal{P}$. Then $\rp\rtimes G$ is an almost AF Cantor groupoid with respect to the AF subgroupoid $\raf\rtimes G$.
\end{theorem}
\begin{proof} By Lemma \ref{minimalsemidirectgroupoid}, $C^*_r(\rp\rtimes G)$ is simple. Thus by \cite{Ph05}, Proposition 2.13 it is enough to check Condition 1 of Definition \ref{almostaf}.

Note that $\raf$ is a subgroupoid of $\raf\rtimes G$, and so any graph in $\raf$ is also a graph in $\raf\rtimes G$. Hence if a compact set is thin for $\raf$, it must be thin for $\raf\rtimes G$. Consider the sets
\[
M_r = \{ ((T, T-x), g)\in (\rp\rtimes G  \setminus \raf\rtimes G) \mid \|x\|\leq r\}
\]
Referring to (\ref{Lr}), we have that $M_r = L_r\times G$. 
Notice that $r(M_r) = r(L_r)$. Suppose that $K\subset (\rp\rtimes G \setminus \raf\rtimes G)$ is compact. Then 
\[
K = \bigcup_{g\in G}K_g \hspace{1cm} \textnormal{ where } \hspace{1cm}K_g = K\cap(\rp \setminus \raf)\times\{g\}.
\]
Each of the $K_g$ is compact because $(\rp \setminus \raf)\times\{g\}$ is closed. If $\kappa_1: \rp\rtimes G\to\rp$ is the usual projection, then $\kappa_1(K_g)$ is compact, and hence included in $L_{r_g}$ for some $r_g$. Let $r = \max\{r_g\}$ and consider $M_r$. We have
\[
K = \bigcup_{g\in G}K_g \subset \bigcup_{g\in G}(L_{r_g}\times\{g\}) \subset \bigcup_{g\in G}(L_{r}\times\{g\}) = M_r
\]
giving us that $K\subset M_r$ and thus $r(K) \subset r(M_r) = r(L_r)$. Since $r(L_r)$ is thin for $\raf$, $r(K)$ must also be thin for $\raf$ and hence for $\raf\rtimes G$. Thus Condition 1 of Definition \ref{almostaf} is satisfied, and we have that $\rp\rtimes G$ is an almost AF Cantor groupoid.
\end{proof}

\begin{corollary}
Suppose that $G$ is a finite symmetry group for $(\mathcal{P},\omega)$ acting freely on $\mathcal{P}$. Then the C*-algebra $C^*_r(\rp\rtimes G)$ has real rank zero, stable rank one, and order on its projections is determined by traces.
\end{corollary}
\begin{proof}
This follows from the above theorem together with Theorem \ref{phillipslist}.
\end{proof}

\begin{corollary}
Suppose that $G$ is a finite symmetry group for $(\mathcal{P},\omega)$ acting freely on $\mathcal{P}$. Then the C*-algebra $\Aw\rtimes G$ has real rank zero, stable rank one, and order on its projections is determined by traces.
\end{corollary}
\begin{proof}
This follows from the above corollary together with Proposition \ref{semidirectcrossedproduct}.
\end{proof}

\section{Traces on $\Aw\rtimes G$}\label{tracesection}
We now show that the AF algebra $C^*_r(\raf\rtimes G) \cong AF_\omega\rtimes G$ has a unique trace; this will imply that $\Aw\rtimes G$ has unique trace by Theorem \ref{phillipslist}. To do this, we show that $AF_\omega\rtimes G = \overline{\cup A_n\rtimes G}$ and that the matrix of partial multiplicities for the inclusion $A_n\rtimes G \hookrightarrow A_{n+1}\rtimes G$ is primitive and does not depend on $n$. 

We first describe the sort of finite group actions we encounter when studying the crossed product of $AF_\omega$ by a finite symmetry group $G$.
\begin{defn}\label{transpermute}
Let $n\geq 1$ and $k\geq 2$ be integers and let $A$ be the finite dimensional algebra
\[
A = \bigoplus_{i=1}^n\mathbb{M}_k = C(I) \otimes \mathbb{M}_k
\]
where $I = \{1, 2, \dots, n\}$. Let $G$ be a finite group and let $\alpha: G\to \textnormal{Aut}(A)$ be a homomorphism. Then we say $\alpha$ {\em (freely and) transitively permutes the summands of $A$} if the restriction of $\alpha$ on $C(I)$ acts by (freely and) transitively permuting $I$.
\end{defn}
For $A$ as in Definition \ref{transpermute}, let $q_i = \chi_{\{i\}}\otimes 1$. We note that if $\alpha$ freely and transitively permutes the summands of $A$, then $\#G = n$ necessarily, with $G = \{e=g_1, g_2,\dots, g_n\}$.
\begin{lemma}
Suppose that $A$, $G$ and $\alpha$ are as in Definition \ref{transpermute}, and that $\alpha$ freely and transitively permutes the summands of $A$. Then there exists a $*$-isomorphism $$\Phi:\mathbb{M}_{\#G}\otimes q_1A\to A\rtimes_\alpha G$$
such that $$e_{ij}\otimes q_1a \mapsto q_i\alpha_{g_i}(a)\delta_{g_ig_{j}^{-1}}.$$
\end{lemma}
\begin{proof}
This is straightforward verification.
\end{proof}

Recall that 
\[
AF_\omega = \overline{\bigcup_{n\in\NN}A_n},
\]
where
\[
A_n = \bigoplus_{p\in\mathcal P} A_{n,p},
\]
with
\begin{eqnarray*}
A_{n,p}&=& \textnormal{span}_{\mathbb{C}}\, \{ e_p^n(x,y)\mid x,y\in \textnormal{Punc}(n,p)\}\\
       &\cong&\mathbb{M}_{\#\text{Punc}(n,p)}.
\end{eqnarray*}
Let $G$ be a finite symmetry group for $(\mathcal{P},\omega)$, and as before let $\tilde\alpha$ denote the action induced by $G$. For $g\in G$ we have
\begin{eqnarray*}
\tilde\alpha_g(e^n_p(x,y))(T, T') &=& e^n_p(x,y)(g^{-1}T, g^{-1}T')\\
                     &=& \left\{\begin{array}{ll} 1 \hspace{0.5cm}& (g^{-1}T, g^{-1}T') \in E^n_p(x,y)\\ 0 & \textnormal{otherwise}\end{array}\right. \\
                     &=& \left\{\begin{array}{ll} 1 \hspace{0.5cm}& (g^{-1}T, g^{-1}T') = (\omega^n(S) - x, \omega^n(S) - y), \\
                                        & \hspace{0.3cm}S\in U(\{p\}, p)\\ 0 & \textnormal{otherwise}\end{array}\right. \\
                     &=& \left\{\begin{array}{ll} 1 \hspace{0.5cm}& (T, T') = (\omega^n(gS) - gx, \omega^n(gS) - gy), \\
                                        & \hspace{0.3cm}gS\in U(\{gp\}, gp)\\ 0 & \textnormal{otherwise}\end{array}\right. \\
                     &=& e^n_{gp}(gx, gy)(T,T').
\end{eqnarray*}
By definition of $A_{n,p}$, we then have $\tilde\alpha_g(A_{n,p}) = A_{n,gp}$. Let $\mathcal{S}_G$ be a set of standard position prototiles for $G$ and assume that $G$ acts freely on $\mathcal{P}$. Because $\mathcal{P} = G\mathcal{S}_G$, we have
\[
A_n = \bigoplus_{p\in\mathcal{S}_G} \left(\oplus_{g\in G} A_{n, gp}\right) \cong \bigoplus_{p\in\mathcal{S}_G}\left( C(Gp)\otimes \mathbb{M}_{\#\text{Punc}(n,p)}\right),
\]
where $Gp$ is the finite set $\{gp\mid g\in G\}$. Let $B_{n,p} = \oplus_{g\in G} A_{n, gp} \cong C(Gp)\otimes \mathbb{M}_{\#\text{Punc}(n,p)}$. The action $\tilde\alpha$ is free and transitive on the set $Gp$, that is, $\tilde\alpha$ restricted to each of the $B_{n,p}$ transitively permutes the summands of $B_{n,p}$ in the sense of Definition \ref{transpermute}. Hence we have
\begin{eqnarray*}
A_n\rtimes_{\tilde\alpha} G &=& \left(\bigoplus_{p\in\mathcal{S}_G} B_{n,p}\right)\rtimes_{\tilde\alpha} G\\
             &=& \bigoplus_{p\in\mathcal{S}_G} \left(B_{n,p}\rtimes_{\tilde\alpha} G\right)\\
             &\cong& \bigoplus_{p\in\mathcal{S}_G}\mathbb{M}_{\#G\cdot\textnormal{Punc}(n,p)}.
\end{eqnarray*}

\begin{theorem}\label{freeafcrossedproducts}
Let $G$ be a symmetry group for $(\mathcal{P},\omega)$ which acts freely on $\mathcal P$. Then $$AF_\omega\rtimes G \cong \overline{\bigcup_{n\in\NN} A_n \rtimes G}.$$ The number of summands in the finite dimensional algebras $A_n\rtimes G$ is the number of elements of $\mathcal{S}_G$, and if $M$ is the incidence matrix for the unital inclusion $A_n \rtimes G \subset A_{n+1} \rtimes G$, then $M_{ij}$ is the number of images of $p_j$ under the action of $\re^d\rtimes G$ in $\omega(p_i)$.
\end{theorem}
\begin{proof}
The isomorphism follows from the fact that crossed products commute with direct limits, and the second statement is by the discussion directly above the theorem.

We denote the inclusion of $A_n$ in $A_{n+1}$ by $\iota$ and the induced inclusion from $A_n\rtimes G$ to $A_{n+1}\rtimes G$ by $I$, that is, $I(a\delta_g) = \iota(a)\delta_g$. We now find the incidence matrix of the inclusions. To do this, we use Equation (\ref{incidenceformula}). Let $q_{n,p}$ denote the identity of $A_{n,p}$. Then the identity of the $p$th summand of $A_n\rtimes G$ is 
\[
\sum_{g\in G} q_{n,gp}\delta_e.
\]
The trace on the $p$th summand is
\[
\tau^{A_n\rtimes G}_p(a\delta_g) = \left\{\begin{array}{ll}\textnormal{Tr}\left(a\sum_{h\in G}q_{n, hp}\right)&  \textnormal{ if } g = e\\0 & \textnormal{ otherwise.}\end{array}\right.
\]
And we have
\begin{eqnarray*}
I\left(\sum_{g\in G} q_{n,gp}\delta_e\right) &=& \iota\left(\sum_{g\in G} q_{n,gp}\right)\delta_e\\
                                             &=& \sum_{g\in G}\iota\left(q_{n,gp}\right)\delta_e
\end{eqnarray*}
Thus for $p_i, p_j\in \mathcal{S}_G$ we have
\begin{eqnarray*}
\tau^{A_{n+1}\rtimes G}_{p_i}\circ I\left(\sum_{g\in G} q_{n,gp_j}\delta_e\right) 
            &=& \tau^{A_{n+1}\rtimes G}_{p_i}\left(\sum_{g\in G} \iota(q_{n,gp})\delta_e\right)\\
            &=& \textnormal{Tr}\left(\left(\sum_{h\in G} q_{n+1, hp_i} \right)\left(\sum_{g\in G} \iota(q_{n, gp_j}) \right)\right)\\
            &=& \sum_{h\in G}\sum_{g\in G} \textnormal{Tr}\left(q_{n+1, hp_i}\iota(q_{n, gp_j})\right)
\end{eqnarray*}
The term $\textnormal{Tr}\left(q_{n+1, hp_i}\iota(q_{n, gp_j})\right)$ is the number of translates of $gp_j$ in $\omega(hp_i)$, by the discussion in Section \ref{tilingalgebra} after Equation (\ref{identity}). Hence
\begin{eqnarray*}
\tau^{A_{n+1}\rtimes G}_{p_i}\circ I\left(\sum_{g\in G} q_{n,gp_j}\delta_e\right) 
            &=& \sum_{h\in G}\sum_{g\in G}\#\textnormal{Punc}(n,p_j)\left(\begin{array}{c}\textnormal{\# of translates of}\\gp_j\textnormal{ in }\omega(hp_i)\end{array}\right)\\
            &=& \#\textnormal{Punc}(n,p_j)\sum_{h\in G}\sum_{g\in G}\left(\begin{array}{c}\textnormal{\# of translates of}\\h^{-1}gp_j\textnormal{ in }\omega(p_i)\end{array}\right)\\
\end{eqnarray*}
For fixed $h$, $\sum_{g\in G}\left(\#\textnormal{ of translates of }h^{-1}gp_j\textnormal{ in }\omega(p_i)\right)$ is the number of images of $p_j$ under the action of $\re^d\rtimes G$ in $\omega(p_i)$. Hence
\begin{eqnarray*}
\tau^{A_{n+1}\rtimes G}_{p_i}\circ I\left(\sum_{g\in G} q_{n,gp_j}\delta_e\right) 
            &=& \#\textnormal{Punc}(n,p_j)\sum_{h\in G}\left(\begin{array}{c}\textnormal{\# of images of }p_j\textnormal{ under the}\\\textnormal{action of }\re^d\rtimes G\textnormal{ in }\omega(p_i)\end{array}\right)\\
            &=& \#G\#\textnormal{Punc}(n,p_j)\left(\begin{array}{c}\textnormal{\# of images of }p_j\textnormal{ under the}\\\textnormal{action of }\re^d\rtimes G\textnormal{ in }\omega(p_i)\end{array}\right)\\
\end{eqnarray*}
On the other hand, 
\[
\tau^{A_n\rtimes G}_{p_j}\left(\sum_{g\in G} q_{n,gp_j}\delta_e\right) = \#G\#\textnormal{Punc}(n,p_j)
\]
because $\sum_{g\in G} q_{n,gp_j}\delta_e$ is the identity on the $p_j$th summand, $\tau^{A_n\rtimes G}_{p_j}$ is the matrix trace restricted to the $p_j$th summand, and the size of the $p_j$th summand is $\#G\#\textnormal{Punc}(n,p_j)$. Hence
\[
\frac{\tau^{A_{n+1}\rtimes G}_{p_i}\circ I\left(\sum_{g\in G} q_{n,gp_j}\delta_e\right)}{\tau^{A_n\rtimes G}_{p_j}\left(\sum_{g\in G} q_{n,gp_j}\delta_e\right)} = \left(\begin{array}{c}\textnormal{\# of images of }p_j\textnormal{ under the}\\\textnormal{action of }\re^d\rtimes G\textnormal{ in }\omega(p_i)\end{array}\right).
\]
Thus by Equation (\ref{incidenceformula}), the incidence matrix of the inclusion is as given in the statement of the theorem.
\end{proof}
Notice that the incidence matrix does not depend on $n$. In fact, it is the same for each inclusion, just as it is for $AF_\omega$. Primitivity of the substitution implies primitivity of the incidence matrix for $AF_\omega\rtimes G$.
\begin{corollary}\label{uniquetracecor}
Let $G$ be a finite symmetry group for $(\mathcal{P},\omega)$ such that $G$ acts freely on the prototiles. Then both $AF_\omega\rtimes G$ and $\Aw\rtimes G$ have unique traces.
\end{corollary}
\begin{proof}
As stated before, that $AF_\omega\rtimes G$ has a unique tracial state is a general fact about AF algebra with constant primitive substitution matrix, again see \cite{Ha81}, Theorem 4.1. By Theorem \ref{phillipslist}, this implies that $\Aw\rtimes G$ has a unique trace.
\end{proof}
\begin{example} Penrose tiling, $G= D_{10}$.\label{penroseafincidence}

Referring to Figure \ref{penrosesub}, we set $\mathcal S_{D_{10}} = \{ \mathbf{1}, \mathbf{21}\}$. Then $\omega(\mathbf{1})$ contains one image each of $\mathbf{1}$ and $\mathbf{21}$ under the action of $\re^2\rtimes D_{10}$, and $\omega(\mathbf{21})$ contains one image of $\mathbf{1}$ and two images of $\mathbf{21}$ under the action of $\re^2\rtimes D_{10}$. Hence the incidence matrix for $AF_\omega \rtimes D_{10}$ is
\[
M = \left[ \begin{array}{cc}1&1\\1&2\end{array}\right].
\]
\end{example}
\begin{rmk}\label{connesremark}
Using standard methods (see \cite{Dav}, Example IV.3.5 for example), one finds from Example \ref{penroseafincidence} that $K_0(AF_\omega\rtimes D_{10})\cong \bz + \phi^{-1}\bz$, where $\phi$ is the golden ratio. This is in fact the same ordered group that Connes obtains in \cite{CoNCG}, Section 2.3 for an AF algebra arising from the space of Penrose tilings, and so his AF algebra must be isomorphic to $AF_\omega\rtimes D_{10}$. In his example, he considers a space homeomorphic to $\op$ and declares two tilings to be equivalent if one can be carried to the other by any isometry of the plane. It may seem odd that we get the same result, since one would imagine that equivalence by any isometry could be bigger than the equivalence relation $\raf \rtimes D_{10}$. However, it is a fact for Penrose tilings that if $(T, T')\in\rp\setminus\raf$, then there exists $g\in D_{10}$ such that $(T, gT')\in \raf$, see for example \cite{KelNCG}, Section 4.2.1. This means that $\raf \rtimes D_{10}$ is in fact equivalence by any isometry on $\op$.
\end{rmk}
\section{The Rokhlin Property and Weak Rokhlin Property}\label{RPsection}

In this section we prove that if $G$ is a symmetry group for $(\p,\omega)$ which acts freely on $\p$, then the induced action of $G$ on $AF_\omega$ has the Rokhlin property. We also show that the action of $G$ on $\Aw$ has the weak Rokhlin property. We first recall the relevant definitions. In what follows, $\mathfrak{T}(A)$ denotes the space of normalized traces on a C*-algebra $A$. 

\begin{defn}\label{rokdef}(\cite{Ph09}, Definition 2.1)
Let $A$ be a unital C*-algebra, and let $\alpha: G\to \textnormal{Aut}(A)$ be an action of a finite group $G$ on $A$. Then $\alpha$ has the {\em Rokhlin property} if for every finite set $\mathcal F\subset A$ and every $\varepsilon >0$ there are mutually orthogonal projections $e_g\in A$ for $g\in G$ such that:
\begin{enumerate}
\item $\left\|\alpha_g(e_h)- e_{gh}\right\| <\varepsilon$ for all $g, h\in G$.
\item $\left\|e_gf- fe_g\right\| < \varepsilon$ for all $g\in G$ and all $f\in \mathcal F$.
\item $\sum_{g\in G}e_g = 1$.
\end{enumerate}
The $(e_g)_{g\in G}$ are called a {\em family of Rokhlin projections} for $\alpha, \mathcal F$ and $\varepsilon$.
\end{defn}

\begin{defn}\label{WRP}(\cite{MaSa12}, Definition 2.7 (4)) 
Let $A$ be a unital simple C*-algebra such that $\mathfrak{T}(A)$ is nonempty, and let $\alpha: G \to$ Aut$(A)$ be an action of a finite group $G$ on $A$. Then $\alpha$ has the {\em weak Rokhlin property} if there exists a central sequence $(f_n)_{n\in\NN}$ in $A$ with $0\leq f_n \leq 1$, 
\[
\lim_{n\to \infty}\|\alpha_g(f_n)\alpha_h(f_n)\| = 0
\] 
for all $g,h\in G$ with $g\neq h$ and
\[
\lim_{n\to \infty}\max_{\tau\in \mathfrak T(A)}|\tau(f_n) - (\#G)^{-1}| = 0.
\]

\end{defn}
As stated in \cite{MaSa12}, this is a weakening of the Rokhlin property suited to cases where the algebra in question may not have many projections. The C*-algebras considered in this work have many projections, but as we will see we are only able to verify Definition \ref{WRP} for the case of a symmetry group acting on $\Aw$.

Our first result concerning the action of $G$ on $AF_\omega$ will follow from the following more general result:

\begin{proposition}\label{locrepfreeAF}
Suppose that $A = \overline{\cup A_n}$ is a unital AF algebra and that $\alpha: G\to$ Aut$(A)$ is an action of a finite group $G$ on $A$. Suppose that for each $n\in\NN$ there exists a positive integer $k(n)$ such that
\[
A_n = \bigoplus_{i\in I(n)}\left( \bigoplus_{g\in G} \mathbb{M}_{k(n)}\right)
\] 
and that the restriction of $\alpha$ freely and transitively permutes the summands of 

\noindent$\bigoplus_{g\in G} \mathbb{M}_{k(n)}$ in the sense of Definition \ref{transpermute}. Then $\alpha$ has the Rokhlin property. 
\end{proposition}
\begin{proof}
The action $\alpha$ restricted to each $A_n$ clearly has the Rokhlin property, and so it must have the Rokhlin property on their union because one may take the finite set to be contained in some $A_n$ (up to an $\eps > 0$).
\end{proof}
The following theorem is a corollary of Proposition \ref{locrepfreeAF}.
\begin{theorem}
Let $G$ be a symmetry group for $(\mathcal{P},\omega)$, and suppose that $G$ acts freely on $\mathcal{P}$. Then the action of $G$ on $AF_\omega$ has the Rokhlin property.
\end{theorem}
\begin{proof}
By the discussion above Theorem \ref{freeafcrossedproducts}, the action of $G$ on $AF_\omega$ satisfies the conditions of Proposition \ref{locrepfreeAF}.
\end{proof}

To prove our next result, we first need a lemma of Putnam. We note that what follows heavily depends on the assumption from Section 2 that the capacity of the boundary of each prototile is strictly less than $d$. In what follows, $\partial(X)$ denotes the topological boundary of the space $X$.
\begin{lemma}(\cite{Pu00}, Lemma 2.3)\label{perimetervolume}
Let $(\mathcal{P}, \omega)$ be a substitution tiling system. For $p\in\mathcal{P}$, $n\in\NN$, and $x\in$ Punc$(n,p)$, we define $D(x)$ to be 
\[
D(x) = \inf\{ \|x - y\| \mid y\in \partial(\textnormal{supp}(\omega^n(p)))\}. 
\]
Then the quotient
\begin{equation}\label{putnamquotient}
\frac{\#\{x\in\textnormal{Punc}(n,p)\mid D(x) < R\} }{\#\textnormal{Punc}(n,p)}
\end{equation}
converges to 0 as $n$ goes to infinity.
\end{lemma}
Intuitively, for a tiling of $\re^2$, (\ref{putnamquotient}) converges to 0 because the numerator scales with the perimeter of a tile while the denominator scales with the area. This allows us to build positive elements which approximately commute with the generating set
\[
\mathcal{E}_2 = \{ e_{t_1t_2}\mid \textnormal{Int}(t_1\cup t_2)\textnormal{ is connected}\}
\]
of $\Aw$ and have large trace.  In the interest of readability, for the remainder of this section we will write $\|a\|$ for the reduced norm of $a$ for any $a\in \Aw$. 
\begin{lemma}\label{trpgenerators}
Let $(\mathcal{P}, \omega)$ be a substitution tiling system. For any $\varepsilon >0$ we can find a positive element $a_g\in \Aw$ for each $g\in G$ with $0\leq a_g \leq 1$ such that
\begin{enumerate}
\item $a_ga_h = 0$ for each $g, h\in G$ with $g\neq h$,
\item $\|\alpha_g(a_h) - a_{gh}\| < \varepsilon$ for all $g, h\in G$,
\item $\|a_gf - fa_g\| < \varepsilon$ for all $g\in G$ and $f\in \mathcal E_2$, and
\item $\tau\left(1 - \sum_{g\in G} a_g\right) < \varepsilon$, where $\tau$ is $\Aw$'s unique trace. 
\end{enumerate}
\end{lemma}
\begin{proof}
For $x\in \textnormal{Punc}(n,p)$ let $t(x)$ be the tile such that $x\in t(x)$, and if $X$ is a set of punctures denote the set of tiles with elements of $X$ as punctures by $t(X)$. We define, for $k\geq 0$, collections of punctures $\rho^k(n,p)$ as follows:
\begin{eqnarray*}
\rho^0(n,p) & = & \{x\in\textnormal{Punc}(n,p)\mid t(x) \cap \partial(\supp(\omega^n(p)))\neq\emptyset\}\\
\rho^k(n,p) & = & \left\{x\in\textnormal{Punc}(n,p)\mid t(x) \cap \left(\partial\left(\supp\left(\omega^n(p)\setminus\cup_{i=0}^{k-1} t(\rho^i(n,p)) \right)\right)  \right)\neq\emptyset\right\}
\end{eqnarray*}
Here, $\partial(A)$ denotes the topological boundary of the set $A\subset \re^d$. Loosely speaking, $\rho^0(n,p)$ is the set of punctures around the edge of the patch $\omega^n(p)$, $\rho^1(n,p)$ is the set of punctures around the edge of the patch $\omega^n(p)$ after removing the outer layer, and so on. 
\begin{center}
\includegraphics[width=\textwidth]{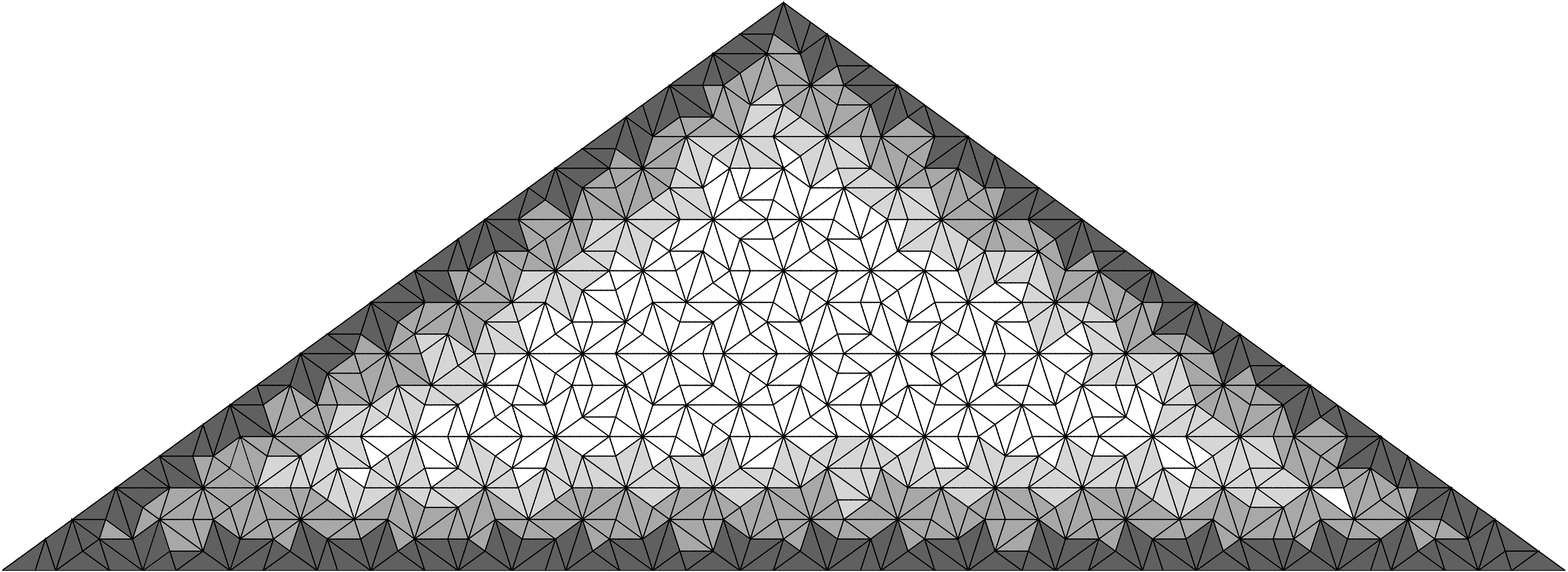}
\end{center}
In this picture, the punctures of the darkest tiles are $\rho^0(n,p)$, the punctures of the next darkest are $\rho^1(n,p)$ and the punctures of the lightest gray are $\rho^2(n,p)$. We notice that there exists $k'\in\NN$ such that for all $k>k'$ we have that $\rho^k(n,p)$ is empty.

Let $\varepsilon >0$ and find $N\in\NN$ such that $N>\frac{2}{\varepsilon}$. Let $\mathfrak{d}$ be the maximum diameter among the prototiles, and let $R>0$ be such that $R> 2N\mathfrak{d}$. By Lemma \ref{perimetervolume}, there exists $s\in\NN$ such that 
\begin{equation}\label{smalltrace}
\frac{\#\{x\in\textnormal{Punc}(s,p)\mid D(x) < R\} }{\#\textnormal{Punc}(s,p)} < \varepsilon.
\end{equation}
Since punctures can be no more than $2\mathfrak{d}$ apart by the triangle inequality, for all $i$ with $0\leq i \leq N$ and $x\in\rho^i(s,p)$ we must have that $D(x)< R$. Let $b_k$ be numbers such that $b_0 = 0$, $0 < b_j - b_{j-1} < \frac{\varepsilon}{2}$, and $b_k = 1$ for all $k > N$. We may find these numbers because $N\cdot\frac{\varepsilon}{2} > 1$. For the identity element $e\in G$, let 
\[
a_e = \sum_{k=0}^\infty \left(b_k \sum_{p\in\mathcal{S}_G\atop x\in\rho^k(s,p)} e^s_p(x,x)\right),
\]
and more generally
\[
a_g = ga_e= \sum_{k=0}^\infty \left(b_k \sum_{p\in\mathcal{S}_G\atop x\in\rho^k(s,p)} e^s_{gp}(gx,gx)\right).
\]
Notice that these sums are finite because $\rho^k(s,p)$ are eventually empty, and notice also that each $a_g$ is an element of the finite dimensional algebra $A_s$. In addition, each $a_g$ is a real-valued function in $C(\op)\subset C^*_r(\rp)$ which takes values between 0 and 1. Hence both $a_g$ and $1-a_g$ have positive square roots, and so $0\leq a_g\leq 1$ for all $g\in G$.

Let us pause for a moment to give an intuitive description of the $a_g$. If we think of them as functions on Punc$(s,p)$, they take the value 0 on the punctures around the boundary of $\omega^s(p)$, they take the value 1 on most of the punctures in the interior, and the values increase gradually from 0 to 1 as we move from the boundary towards the middle. The values that the $a_g$ take on punctures whose tiles share an edge always differ by less than $\varepsilon$. Furthermore, they only take values less than 1 in a relatively small band of punctures near the boundary.

Since $a_g$ is supported on the diagonal of $\rp$ for all $g\in G$, we have that for any given $T\in \op$, $a_g(T,T')$ is only possibly nonzero if $T'=T$, and so by Equation (\ref{rsnorms}), 
\[
\|a_g\|_r = \sup_{T\in \op} \left\{\sum_{T'\in [T]}|a_g(T,T')|\right\} = \sup_{T\in \op} |a_g(T,T)| = 1,
\]
\[
\|a_g\|_s = \sup_{T\in \op} \left\{\sum_{T'\in [T]}|a_g(T',T)|\right\} = \sup_{T\in \op} |a_g(T,T)| = 1.
\]

Since $\|a_g\|_I$ is the max of these two norms and $\|a_g\|_I$ dominates the reduced norm, we have $\|a_g\|\leq1$.  The $a_g$ elements satisfy Condition 1 trivially, and from the previous section we see that $ga_h = a_{gh}$.  

To prove Condition 3, recall that $e^s_p(x,x)$ is the characteristic function of the set $E^s_p(x,x)$, which is a compact open subset of the unit space. Take 
\[
q = e_{t_1t_2}\in \mathcal E_2,
\]
and calculate
\[
(a_gq)(T, T') = \sum_{k=0}^\infty \left(b_k \sum_{p\in\mathcal{S}_G\atop x\in\rho^k(s,p)} e^s_{gp}(gx,gx)q(T, T')\right).
\] 
We have
\begin{eqnarray*}
e^s_{gp}(gx,gx)q(T, T') &=& \begin{cases} q(T,T') & \text{if }T\in E^s_{gp}(gx,gx)\\ 0 & \text{otherwise}\end{cases}\\
                                       &=& \begin{cases} 1 & \text{if }(T,T') \in V(\{t_1, t_2\},t_1, t_2)\text{ and}\\ & T\in E^s_{gp}(gx,gx)\\ 0 & \text{otherwise.}\end{cases}
\end{eqnarray*}
Given $T\in\op$, there exist unique $g\in G$, $p\in \mathcal{P}$ and $x\in$ Punc$(s,p)$ such that $T\in E^s_{gp}(gx,gx)$. This puncture $x$ must be an element of  $\rho^k(s,p)$ for some $k$. Then in this case we have
\[
(a_gq)(T, T') = \left\{\begin{array}{ll} b_k & \text{if }(T,T') \in V(\{t_1, t_2\},t_1, t_2)\\ 0 & \textnormal{otherwise.}\end{array}\right.
\] 
Now we calculate $qa_g$:
\[
(qa_g)(T, T') = \sum_{k=0}^\infty \left(b_k \sum_{p\in\mathcal{S}_G\atop x\in\rho^k(s,p)} qe^s_{gp}(gx,gx)(T, T')\right),
\] 
and similar to above
\begin{eqnarray*}
qe^s_{gp'}(gy,gy)(T, T') &=& \left\{\begin{array}{ll} q(T,T') & \text{if }T'\in E^s_{gp'}(gy,gy),\\ 0 & T'\notin E^s_{gp'}(gy,gy)\end{array}\right.\\
                                             &=& \left\{\begin{array}{ll} 1 & \text{if }(T,T') \in V(\{t_1, t_2\},t_1, t_2)\textnormal{ and}\\ & T'\in E^s_{gp}(gy,gy)\\ 0 & \textnormal{otherwise.}\end{array}\right.
\end{eqnarray*}
As above, given $T'\in\op$ there exist unique $g\in G$, $p\in \mathcal{P}$ and $y\in$ Punc$(s,p)$ such that $T\in E^s_{gp}(gy,gy)$. This puncture $y$ must be an element of  $\rho^m(s,p)$ for some $k$. Then in this case we have
\[
(qa_g)(T, T') = \left\{\begin{array}{ll} b_m & \text{if }(T,T') \in V(\{t_1, t_2\},t_1, t_2)\\ 0 & \textnormal{otherwise.}\end{array}\right.
\] 
Hence, we may calculate the difference
\begin{eqnarray*}
(a_gq - qa_g)(T,T') &=& \left\{\begin{array}{ll} b_k-b_m & (T,T') \in V(\{t_1, t_2\},t_1, t_2),\\
                                                         & T\in E^s_{gp}(gx,gx), x\in \rho^k(s,p), \textnormal{ and} \\
                                                         & T'\in E^s_{gp'}(gy,gy),y\in \rho^m(s,p')\\ 
                                                       0 & \textnormal{otherwise.}\end{array}\right.
\end{eqnarray*}
If we are in the first case and $p\neq p'$, then $k$ and $m$ must both be zero. Indeed, if $p\neq p'$, then $\{t_1, t_2\}$ is a two-tile pattern whose edge lies along the boundary of $gp$ and $gp'$, and hence $t(x)$ and $t(y)$ intersect the boundaries of $\omega^s(p)$ and $\omega^s(p')$ respectively. Thus $x\in\rho^0(s,p)$ and $y\in\rho^0(s,p')$. In the case where $p=p'$, the conditions in the first case above imply that the patch $\{t(gx), t(gy)\}$ is a translate of $\{t_1, t_2\}$. Hence the difference between $k$ and $m$ is at most 1, and by the definition of the $b_i$ this implies that $|b_k-b_m|< \frac{\varepsilon}{2}$. Furthermore, if $T\in \op$, there is at most one $T'$ for which $(a_gq - qa_g)(T,T')$ is nonzero, namely $T' = T + x_{t_1} - x_{t_2}$ if $T$ happens to be in $U(\{t_1, t_2\}, t_1)$. Hence 
\[
\|a_gq - qa_g\|_{r} = \sup_{T\in\op}\left\{\sum_{T'\in [T]}|(a_gq - qa_g)(T,T')|\right\} \leq \frac{\varepsilon}{2}
\]
\[
\|a_gq - qa_g\|_{s} = \sup_{T'\in\op}\left\{\sum_{T\in [T']}|(a_gq - qa_g)(T,T')|\right\} \leq \frac{\varepsilon}{2}
\]
\[
\|a_gq - qa_g\| \leq \max\{\|a_gq - qa_g\|_{r}, \|a_gq - qa_g\|_{s}\} \leq \frac{\varepsilon}{2} < \varepsilon.
\]
Hence Condition 3 is satisfied. To prove Condition 4 we use Equation (\ref{smalltrace}). The function $1- \sum_{g}a_g$ is nonnegative and is only nonzero on elements $(T,T)$ such that 
\[
(T,T)\in \bigcup_{p\in\mathcal P}\bigcup_{i=0}^{N-1}\bigcup_{x\in\rho^i(s,p)} E^s_p(x,x).
\]
Notice that the above union is a disjoint union. Hence for every $T\in\op$ we have that 
\[
 \left(1- \sum_{g}a_g\right)(T,T) \leq \sum_{p\in\mathcal{P}}\sum_{i=0}^{N-1}\sum_{x\in\rho^i(s,p)}e^s_p(x,x)(T,T).
\]
We now calculate the value of our trace $\tau$ on this element. We have
\begin{eqnarray*}
\tau\left(1- \sum_{g}a_g\right) &=& \sum_{p\in\mathcal{P}}\sum_{i=0}^{N-1}\sum_{x\in\rho^i(s,p)}\tau(e^s_p(x,x))\\
                     &=& \sum_{p\in\mathcal{P}}\sum_{i=0}^{N-1}\sum_{x\in\rho^i(s,p)}\lambda^{-ds}v_L(p)\\
                     &\leq& \sum_{p\in\mathcal{P}}\#\{x\in\textnormal{Punc}(s,p)\mid D(x) < R\}\lambda^{-ds}v_L(p)\\
                     &<&\varepsilon\sum_{p\in\mathcal{P}}\#\textnormal{Punc}(s,p)\lambda^{-ds}v_L(p)\\
                     &=&\varepsilon
\end{eqnarray*}
where the last line is by Equation (\ref{tracialstateformula}).
\end{proof}
We note that the $a_g$ in the above proof are constructed using a technique similar to that seen in the proof of \cite{Th10}, Theorem 4.32.

Since $\mathcal{E}_2$ is a generating set for $\Aw$, we have the following.

\begin{theorem}\label{AwWRP}
Let $(\mathcal{P}, \omega)$ be a substitution tiling system and suppose that $G$ is a symmetry group for $(\mathcal{P},\omega)$ that acts freely on $\mathcal P$. Then the action of $G$ on $\Aw$ has the weak Rokhlin property.
\end{theorem}
\begin{proof}
It is straightforward from the above that if we take any sequence of positive numbers $\eps_n \to 0$ and apply the proof of Lemma \ref{trpgenerators} for $\eps_n$ to obtain $\{a^{(n)}_g\}_{g\in G}$, then $(a^{(n)}_e)_{n\in\NN}$ is a central sequence satistfying the conditions of Definition \ref{WRP}.
\end{proof}

The rest of this section concerns the {\em tracial rank zero} property, which was first defined by Lin in \cite{Lin01}. Tracial rank can be seen as a noncommutative analogue of topological dimension, see \cite{Lin2}. We will not give the general definition, but we state an equivalent form in the context of simple C*-algebras which was proved by Lin in \cite{Lin2} and which was stated as given by Brown in \cite{Br06}. In what follows, for C*-algebras $A$ and $B$ with $B\subset A$, $a\in A$ and $\varepsilon >0$, the notation $a \in_\varepsilon B$ means $\inf\{ \|b - a\| \mid b\in B\} < \varepsilon$.
\begin{theorem}(\cite{Br06}, Theorem 4.5.1) \label{rokbrown} Let $A$ be a simple C*-algebra. Then $A$ has tracial rank zero if and only if $A$ has real rank zero, stable rank one, the order of projections on $A$ is determined by traces, and for every finite subset $\mathcal F \subset A$ and $\varepsilon>0$ there exists a finite dimensional subalgebra $F\subset A$ with $p = 1_F$ such that:
\begin{enumerate}
\item $\|pf-fp\| < \varepsilon$ for all $f\in \mathcal F$,
\item $pfp \in_\varepsilon F$ for all $f\in \mathcal F$, and
\item $\tau(p) > 1-\varepsilon$ for $\tau\in\mathfrak{T}(A)$. 
\end{enumerate}
\end{theorem}
The following is a corollary of Theorem \ref{AwWRP}.
\begin{corollary}\label{tracialrankzerocor}
Let $(\mathcal{P}, \omega)$ be a substitution tiling system and suppose that $G$ is a symmetry group for $(\mathcal{P},\omega)$ that acts freely on $\mathcal P$. If $\Aw$ has tracial rank zero, then $\Aw\rtimes G$ has tracial rank zero.
\end{corollary}
\begin{proof}
By Theorem \ref{AwWRP} and \cite{MaSa12}, Remark 2.8, the action of $G$ on $\Aw$ is strongly outer (\cite{MaSa12}, Definition 2.7 (2)). Since $\Aw$ has a unique trace, \cite{MaSa12} Theorem 5.1 implies that the action of $G$ on $\Aw$ has the tracial Rokhlin property of Phillips (\cite{Ph09}, Definition 3.1). Then by \cite{Ph11}, Theorem 2.6, $\Aw\rtimes G$ has tracial rank zero.
\end{proof}

We conclude with remarks concerning the tracial rank zero property.

\begin{rmk} Our first remark concerns Corollary \ref{tracialrankzerocor}. In \cite{Ph05}, Question 8.1, Phillips asks the question of whether every C*-algebra of an almost AF Cantor groupoid has tracial rank zero. If the answer to this question is yes, then it would appear that Corollary \ref{tracialrankzerocor} would follow immediately, as both $\rp$ and $\rp\rtimes G$ are almost AF Cantor groupoids. However, it is a fact that if $\rp$ is the groupoid formed from any tiling of $\re^d$ consisting of polytopes which meet full-edge to full-edge which has repetitivity and strong aperiodicity, then there exists a free minimal transformation group groupoid $(X, \bz^d)$ with $X$ homeomorphic to the Cantor set and a clopen $U\subset X$ such that $\rp$ is isomorphic to $(X, \bz^d)^U_U$, see \cite{SW03}. There is no such result for the groupoid $\rp\rtimes G$. Hence Corollary \ref{tracialrankzerocor} would tell us that $\Aw\rtimes G$ has tracial rank zero if one could prove that $C(X)\rtimes \bz^d$ has tracial rank zero for all free minimal actions of $\bz^d$ on the Cantor set $X$.
\end{rmk}
\begin{rmk}
In the author's PhD thesis \cite{StThesis}, it was proved directly (using Lemma \ref{trpgenerators}) that if one assumes that $\Aw$ has tracial rank zero, then the action of $G$ on $\Aw$ has the tracial Rokhlin property. From this, Corollary \ref{tracialrankzerocor} was obtained. 
\end{rmk}
\begin{rmk}
For $\mathcal{F} = \mathcal{E}_2 \subset \Aw$ and $\eps >0$, the elements $a_g\in \Aw$ in the proof of Lemma \ref{trpgenerators} are such that if we set $a = \sum_{g\in G}a_g$, then $a$ satisfies conditions 1--3 in Theorem \ref{rokbrown}. Of course $a$ is not a projection, so we cannot conclude that $\Aw$ is tracial rank zero.
\end{rmk}

{\bf Acknowledgements.} The author wishes to thank Thierry Giordano, Ian Putnam, Michael Whittaker, Daniel Gon\c {c}alves, David Handelman, and Siegfried Echterhoff for many helpful conversations and suggestions regarding this work. In particular, great gratitude is due Thierry Giordano who supervised work on the author's PhD thesis from which this paper is derived. We would also like to thank the referee for his many helpful comments and suggestions.

\end{document}